\documentclass{IEEEtran}
\usepackage{cite}
\usepackage{amsmath,amssymb,amsfonts}
\usepackage{algorithm}

\usepackage{graphicx}
\usepackage{textcomp}
\usepackage{bm}
\usepackage{mathrsfs}
\usepackage{graphicx,cite}
\usepackage{amsmath}
\usepackage{amssymb}
\usepackage[dvips]{epsfig}
\usepackage{latexsym}
\usepackage{psfrag}
\usepackage{graphicx,subfigure}
\usepackage{hyperref}
\usepackage{color}
\usepackage{amsfonts,mathrsfs,amsmath,stmaryrd,amssymb,pifont}%
\usepackage{amscd,graphicx,array,texdraw,tikz}%
\usetikzlibrary{arrows,decorations.markings,shapes,shadows,fadings}
\usepackage{mathrsfs}%
\usepackage{array}
\usepackage{enumerate}
\usepackage{times}
\usepackage{ntheorem}
\usepackage{algorithm}
\usepackage{algpseudocode}
\usepackage{amsmath}
\usepackage{booktabs}  
\usepackage{xcolor}    
\usepackage{pifont}    
\usepackage{multirow}  
\usepackage{amsmath}   
\theoremseparator{.}
\newtheorem{definition}{Definition}
\newtheorem{lemma}{Lemma}
\newtheorem{theorem}{Theorem}
\newtheorem{remark}{Remark}
\newtheorem{assumption}{Assumption}

\newtheorem{proof}{Proof}[section]

\allowdisplaybreaks

\begin{document}

\title{Decentralized online stochastic generalized Nash Equilibrium seeking for multi-cluster games: A Byzantine-resilient algorithm}
\author{Bingqian Liu, Guanghui Wen, \IEEEmembership{Senior Member, IEEE}, Liyuan Chen, Yiguang Hong, \IEEEmembership{Fellow, IEEE}
    \thanks{This work was supported in part by the National Key Research and Development Program of China under Grant No. 2022YFA1004702, and  in part by the National Natural Science Foundation of China through Grant No. 62325304.}
    \thanks{B. Liu, G. Wen and L. Chen are with the Laboratory of Security Operation and
Control for Intelligent Autonomous Systems, Department of Systems
Science, School of Mathematics, Southeast University, Nanjing 211189,
P. R. China (e-mails: bqliu@seu.edu.cn; ghwen@seu.edu.cn; clymath@seu.edu.cn).}
	\thanks{Y. Hong is with the Department of Control Science and Engineering and the Shanghai Research Institute for Intelligent Autonomous Systems, Tongji University, Shanghai 200092, P. R. China (e-mail: yghong@iss.ac.cn).}
}

\maketitle

\begin{abstract}
This paper addresses the challenge of solving the generalized Nash Equilibrium seeking problem for decentralized stochastic online multi-cluster games amidst Byzantine agents. During the game process, each honest agent is influenced by both randomness and malicious information propagated by Byzantine agents. Additionally, none of the agents have prior knowledge about the number and identities of Byzantine agents. Furthermore, the stochastic local cost function and coupled global constraint function are only revealed to each agent in hindsight at each round. One major challenge in addressing such an issue is the stringent requirement for each honest agent to effectively mitigate the effect of decision variables of Byzantine agents on its local cost functions. To overcome this challenge, a decentralized Byzantine-resilient algorithm for online stochastic generalized Nash equilibrium (SGNE) seeking is developed, which combines variance reduction, dynamic average consensus, and robust aggregation techniques. Moreover, novel resilient versions of system-wide regret and constraint violation are proposed as metrics for evaluating the performance of the online algorithm. Under certain conditions, it is proven that these resilient  metrics grow sublinearly over time in expectation. Numerical simulations are conducted to validate the theoretical findings.
\end{abstract}

\section{Introduction}
Generalized Nash equilibrium (GNE) seeking problems have garnered extensive attention in recent years due to their broad applicability across various fields, including smart grids \cite{Deng_2022}, optical networks \cite{Pan_2009}, and economic activity \cite{Sun_2021}. Note that the cost functions and constraint functions that rely on other agents' decisions must be determined. However, most practical problems necessitate the consideration of uncertainties, such as electricity markets with unknown demand \cite{Henrion_2007}, as well as transportation systems with unpredictable travel times \cite{Watling_2006}. Under these circumstances, the SGNE problem has been investigated, where the objectives and constraints of each agent are conceptualized in terms of expectations. This means that uncertainties are mathematically treated as random variables.

Solving SGNE seeking problems is inherently challenging, as the expected values of the objectives and constraints are often inaccessible in exact form or not amenable to efficient computation. A distributed relaxed-inertial forward-backward-forward algorithm for SGNE seeking was proposed in \cite{Cui_2021}, which employs agent-specific mini-batch gradient estimators to address uncertainty. However, due to the limited sample number in each iteration, this approach may suffer from high variance in gradient estimation, which can hinder convergence or slow down the decision-making process. To mitigate this difficulty, variance reduction techniques such as the stochastic variance reduced gradient (SVRG) method \cite{Johnson_2013} and the stochastic average gradient algorithm (SAGA) \cite{Defazio_2014} are commonly utilized. Alternatively, the variance can also be reduced by increasing the number of samples used in each iteration. As demonstrated in \cite{Franci_2022}, this approach is systematically implemented in their proposed stochastic relaxed forward-backward algorithm, which enables distributed SGNE seeking in merely monotone games via a variance-reduced stochastic approximation (SA) scheme. Furthermore, beyond stochasticity, the objectives and constraints may also exhibit temporal variability and may only be revealed in hindsight. This is referred to as the online settings. Specifically, an online multi-cluster game framework was introduced in \cite{Yu_2023}. Nevertheless, the study of online multi-cluster games under stochastic environments remains largely unexplored.

\begin{table*}[t]  
\centering
\caption{Comparison of this paper to some related works involving Byzantine agents}
\label{tab:comparison}
\begin{tabular}{lcccccccc}
\specialrule{1.2pt}{0pt}{3pt} 
References& Problem type & Constraints& Filtering method& Variance reduction & Regret type \\
\specialrule{0.8pt}{0pt}{3pt} 
\cite{Su_2021} & Distributed optimization & \ding{55}& Coordinate-wise trimmed mean & -- & -- \\
\cite{Wu_2020}& Finite-sum optimization &\ding{55}&Geometric median aggregation& SAGA & --\\
\cite{Wu_2023} & Decentralized stochastic optimization  &\ding{55} & Iterative outlier scissor& \ding{55} & --  \\
\cite{Sahoo_2021} & Distributed online optimization & \ding{55} & Coordinate-wise trimmed mean & --&Static regret\\
\cite{Dong_2024} & Distributed online stochastic optimization & \ding{55} & Robust bounded aggregation& \ding{55} & Static regret \\
\cite{Gadjov_2023} & Noncooperative games & \ding{51} & Observation
graph& \ding{55} & -- \\
This paper & Online multi-cluster games  &\ding{51} & Coordinate-wise trimmed mean& SVRG& Dynamic regret  \\
\specialrule{1.2pt}{3pt}{0pt} 
\end{tabular}
\end{table*}

A common underlying assumption in the aforementioned studies is that agents behave reliably. However, in practice, some agents known as Byzantine agents may send malicious messages to their neighbors, causing the decision-making process to be potentially misled. Meanwhile, the honest agents are collectively unaware of the number and identities of the Byzantine agents. Indeed, most of the existing works under Byzantine attacks focus on optimization problems \cite{Su_2021,Karimireddy_2021,Sahoo_2021,Wu_2020,Wu_2023,Dong_2024} rather than game problems \cite{Gadjov_2023,Peng_2024}. The main reason is that the objectives of each agent only depend on itself in the former, and thus, a robust optimization problem can be formulated to prevent Byzantine agents from lying arbitrarily about their own cost functions. However, there are other agents' decision variables in the arguments of the objectives of each agent in the latter. In this case, the influence of Byzantine agents cannot be avoided, even when only the objectives of honest agents are considered.To tackle the issue of how to determine whether the messages received from neighbors of each honest agent is correct, an observation graph was proposed in \cite{Gadjov_2023} to obtain truthful action information. In \cite{Peng_2024}, a centralized operator was utilized to compute aggregation via the geometric median rule for aggregative games under Byzantine attacks.

Motivated by existing research, this paper presents a general framework that integrates SGNE seeking with online multi-cluster games under Byzantine attacks. Specifically, the function information is revealed to each agent only after the decisions have been sequentially committed. At each time step, each honest agent aims to minimize the total expected local cost functions of all honest agents in its cluster, subject to local set constraints and expected coupled global inequality constraints. Additionally, the number and identities of the Byzantine agents are unknown. The main contributions of this paper can be summarized as follows:
 
\begin{itemize}
\item[1)] The decentralized SGNE seeking problem in online multi-cluster games amidst Byzantine attacks is studied for the first time. One of the main difficulties is that the objectives and constraints can be influenced by the decisions of Byzantine agents. To address this challenge, each honest agent constructs estimated variables to decouple  both its local cost functions and global coupling inequality constraints from the potential influence of Byzantine agents.
\item[2)] Utilizing  the modified SVRG (mod-SVRG)  technique, a novel decentralized Byzantine-resilient online SGNE seeking algorithm (DBROSA) is proposed. This algorithm strategically combines the modified SVRG technique, dynamic average consensus, and a robust aggregation mechanism to collectively address malicious behavior, inherent stochasticity, and partial knowledge of objectives and constraints.
\item[3)] The resilient system-wise regret $\mathcal R_{\mathcal H}(T)$ and constraint violation $\mathcal{CV}_{\mathcal H}(T)$, designed to mitigate the effects of Byzantine agents' decisions, are firstly proposed as performance metrics for DBROSA and are proven to exhibit sublinear growth in expectation.
\end{itemize}

A comparison between this paper and related works involving Byzantine agents is summarized in Table \ref{tab:comparison}.

The remainder of this paper is structured as outlined below. Section \ref{sec.one} mainly focuses on the formalization of stochastic online multi-cluster games. Section \ref{sec.two} introduces the DBROSA algorithm. Section \ref{sec.three} presents the main theoretical results concerning the resilient system-wise regret $\mathcal R_{\mathcal H}(T)$ and constraint violation $\mathcal{CV}_{\mathcal H}(T)$. Section \ref{sec.four} provides numerical examples to validate the effectiveness of the proposed method. Finally, section \ref{sec.five} concludes the paper.

{\bf Notations}. Let $\mathbb{R}$, $\mathbb{R}^n$ and $\mathbb{R}^{m \times n}$ represent the collections of real numbers, $n$-dimensional real vectors and $m \times n$-dimensional real matrices, respectively. For any positive integer $N>1$, the set $[N]$ is defined as $\left\{1,\cdots,N\right\}$. For a sequence of column vectors $y_{1},\cdots,y_{N}$, $col(y_{k})_{k\in[N]}\triangleq col(y_{1},\cdots,y_{N})=\left[y_{1}^\top,\cdots,y_{N}^\top\right]^\top$. Given a vector $y \in \mathbb{R}^n$, $\left\|  y \right\|$ signifies its standard Euclidean norm, the $i$th component of $y$ is denoted by ${\left[y \right]}_{i}$, ${{P}_{\chi}}\left[y\right]$ represents the Euclidean projection of $y$ onto a closed convex set $\chi$, and $[y]_{+}\triangleq col(max\left\{0,[y]_{k}\right\})_{k \in [n]}$. For any matrix $W$, ${\left[ W \right]}_{ij}$ denotes the $(i,j)$-th entry of $W$, $[W]_{(k,:)}$ and $[W]_{(:,k)}$ denote the $k$-th row and the $k$-th column of $W$, respectively. $W^{T}$ stands for the transpose of $W$. Particularly, $I_n$ denotes the $n\times n$ identity matrix, and $\bm{1}_n$  is an $n$-dimensional column vector whose entries are all ones. The indicator function $\mathbb I(a\neq b)$ equals $1$ when $a\neq b$ and $0$ otherwise. For a vector-valued function $g(x):\mathbb R^{n}\rightarrow \mathbb R^{m}\triangleq [g_{1}(x),g_{2}(x),\cdots,g_{m}(x)]^\top$, the derivative of $g(x)$ is defined as $\nabla_{x}g(x)\triangleq [\nabla g_{1}(x),\nabla g_{2}(x),\cdots,\nabla g_{m}(x)]^\top \subseteq \mathbb R^{m\times n}$. For a set $\mathcal S$, $\left|\mathcal S\right|$ denotes the cardinality of $\mathcal S$. Given two sets $A$ and $B$, $A\setminus B$ is the set of all elements that belong to $A$ but not to $B$. Given two sequences $\{{{x}_{t}}\}_{t=1}^{\infty }$ and $\{{{y}_{t}}\}_{t=1}^{\infty }$, ${{x}_{t}}=\mathcal{O}({{y}_{t}})$ indicates the existence of $t_{0},c>0$ such that $\left| {{x}_{t}} \right| \le c \left| {{y}_{t}} \right|$ holds for $t \ge t_{0}$. Given $a,b\in\mathbb R$, $\bmod(a,b)$  represents the remainder of $a$ divided by $b$.

\section{Problem Formulation}\label{sec.one}
This section starts with the analysis of the graph structure and subsequently introduces the stochastic online multi-cluster games amidst Byzantine agents. Additionally, several metrics to assess the performance of the online SGNE seeking algorithm are proposed.

\subsection{Graph Structure}\label{sec.one.1}
The agents in the online multi-cluster game interact with each other over a directed graph $\mathcal G=\left(\left[n\right],\mathcal E\right)$, where $\left[n\right]$ represents the collection of all agents and $\mathcal E \subseteq [n]\times[n]$ represents the collection of edges. Assuming that all agents are divided into $N$ clusters, with each cluster containing at least one honest agent. The communication graph among agents in cluster $i$ is depicted by $\mathcal G_{i}=\left(\left[n_{i}\right],\mathcal E_{i}\right)$, where $\left[n_{i}\right]$ and $\mathcal E_{i}\subseteq [n_{i}]\times[n_{i}]$ denote the collection of agents and edges in cluster $i$, respectively. Notably, $n\triangleq \sum_{i=1}^{N}n_{i}$. A directed graph is strongly connected if and only if there is at least one directed path from any agent to all the other agents. Let $\mathcal H$ denote the collection of honest agents and $\mathcal B$ denote the collection of Byzantine agents in graph $\mathcal G$, thus $\mathcal H \cup \mathcal B = \left[n\right]$. For any cluster $i$, $\mathcal H_{i}$ and $\mathcal B_{i}\triangleq [n_{i}]\setminus\mathcal H_{i}$ represent the honest agents and Byzantine agents in graph $\mathcal G_{i}$, respectively. Notably, honest agents have no knowledge of the number or identities of Byzantine agents. Without loss of generality, the upper bound of $\left|\mathcal B_{i}\right|$ is set to $b_{i}$, i.e., $\left|\mathcal B_{i}\right| \leq b_{i}$, thus $\left|\mathcal B\right| \leq b$ with $b\triangleq \sum_{i=1}^{N}b_{i}$. For all $i \in [N]$, $b_{i}$ is well-known piece of information.
\begin{definition}\label{def.1}
 \rm{\cite{Su_2021} A reduced graph $\mathcal R(\mathcal H, \mathcal E^{b-})$ of graph $\mathcal G \left(\mathcal V, \mathcal E \right)$ is defined as a subgraph obtained by removing all the Byzantine agents from $\mathcal V$, and furthermore, by removing any up to $b$ incoming edges of each honest agent.}
\end{definition}
\begin{definition}\label{def.2}
\rm{\cite{Vaidya_2012_1} A source component $\mathcal S \left(\mathcal V^{*},\mathcal E^{*} \right)$ of graph $\mathcal G \left(\mathcal V,\mathcal E \right)$ is a strongly connected subgraph satisfying that each agent in $\mathcal V^{*}$ has a directed path to any other agents in $\mathcal V \backslash \mathcal V^{*}$ and is not reachable from any other agents in $\mathcal V \backslash\mathcal V^{*}$.}
\end{definition}

It is worth noting that a graph containing a source component cannot be strongly connected, since there are no directed paths from the agents outside the source component to the agents within it. To  ensure the robustness of honest agents against misleading messages propagated by Byzantine agents, the following assumption is commonly imposed.

\begin{assumption}\label{assp.2}
\rm{For any $i \in \left[N\right]$, every reduced graph $\mathcal R(\mathcal H_{i},\mathcal E_{i}^{b_{i}-})(\text{or}~\mathcal R(\mathcal H,\mathcal E^{b-}))$ of graph $\mathcal G_{i}(\text{or}~\mathcal G)$ has a nonempty source component.}
\end{assumption}

Assumption \ref{assp.2} guarantees that the honest agents, after applying the CTM mechanism, still retain the ability to propagate information effectively. It also implies that the total number of incoming neighbors for each honest agent in the graph $\mathcal G_i$ (or $\mathcal G$) is at least $2b_i + 1$ (or $2b + 1$) \cite{Su_2021}. For convenience, let $\mathcal N_{ij}^{i}\triangleq \left\{k\in[n_{i}]|(j,k)\in\mathcal E_{i}\right\}$ be the set of incoming neighbors of agent $j$ in cluster $i$ over graph $\mathcal G_{i}$ and $\mathcal N_{ij}\triangleq\left\{(p,q)\in[N]\times [n_{p}]|(\sum_{k=1}^{i-1}n_{k}+j,\sum_{k=1}^{p-1}n_{k}+q)\in\mathcal E\right\}$ be the set of incoming neighbors of agent $j$ in cluster $i$ over graph $\mathcal G$ with $\sum_{k=1}^{0}n_{k}\triangleq 0$.

\subsection{Stochastic Online  Multi-cluster Games}\label{sec.one.3}
If Byzantine agents are absent, each agent $j$ in cluster $i$ has a time-varying local  stochastic cost function $f_{j}^{i,t}\left(x_{i,t},x_{-i,t},\theta_{i,t}\right):\mathbb R^{D+nd}\rightarrow\mathbb R$, where $x_{i,t}=col(x_{ij,t})_{j\in[n_{i}]}\in \Omega_{i}\triangleq \Omega_{i1}\times \cdots \times\Omega_{in_{i}} \subseteq \mathbb R^{n_{i}d}$ is the decision variable of cluster $i$ with $x_{ij,t}$ and $\Omega_{ij}\subseteq \mathbb R^{d}$ denoting the decision variable and the constraint set of agent $j$ in cluster $i$, respectively, $x_{-i,t}=col(x_{1,t},\cdots,x_{i-1,t},x_{i+1,t},\cdots,x_{N,t})$ denotes all the decision variables except for cluster $i$, $\theta_{i,t} \in \Theta_i\subseteq\mathbb R^{D}$ is a random parameter following a time-varying probability distribution $\mathcal P_{i,t}$, and $t \in [T]$ with $T$ denoting a time horizon. The agents in the same cluster $i$ aim to optimize the sum of the expectation of time-varying local stochastic cost functions as 
\begin{equation}\label{eq.problem1}
\frac{1}{n_{i}}\sum_{j =1}^{n_{i}}\mathbf{E}_{\theta_{i,t}}\left[f_{j}^{i,t}\left(x_{i,t},x_{-i,t},\theta_{i,t}\right)\right].
\end{equation}

The cost function of cluster $i$ is coupled with other clusters' decisions. However, not only the cost functions are coupled among different clusters, the stochastic constraint functions $g_{t}:\mathbb R^{D+nd}\rightarrow \mathbb R^{m}$ can also be coupled. Thus, the global coupled inequality constraint function is defined as 
\begin{equation}\label{eq.problem2}
\mathbf{E}_{\omega_{t}}\left[g_{t}\left(x_{t},\omega_{t}\right)\right]\leq \bm 0_{m},
\end{equation}
where $x_{t}=col(x_{1,t},\cdots,x_{N,t})\in \mathbb R^{nd}$ and $\omega_{t}\in \Lambda\subseteq\mathbb R^{D}$ follows a time-varying probability distribution $\mathcal Q_{t}$. Specifically, the cost function $f^{i,t}_{j}$ and constraint function $g_{t}$  are sequentially disclosed to agent $j$ in cluster $i$ only after the decision variable $x_{ij,t} \in \mathbb R^{d}$ has been selected for $\forall t\in[T]$.

If Byzantine agents are present, the malicious messages they send may hinder the optimization process of minimizing the cost function (\ref{eq.problem1}) under constraint (\ref{eq.problem2}) for each cluster $i$, as any Byzantine agent $j \in \mathcal{B}_i$ can arbitrarily manipulate its local cost function \(f_j^{i,t}\). Let $\Omega\triangleq\Omega_{1}\times \cdots \times \Omega_{N}$. In this case, a \textit{stochastic online multi-cluster game} is formulated as follows.
\begin{equation}\label{eq.BROMG}
\forall i \in [N]: \quad
\left\{
\begin{aligned}
\min_{x_t \in \Omega} \quad & F_{i,t}(x_t), \\
\text{s.t.} \quad & G_t(x_t) \leq \bm{0}_m,
\end{aligned}
\right.
\end{equation}
where the cost function is defined as
\begin{equation}\label{eq.costfunction}
F_{i,t}(x_t) \triangleq \mathbf{E}_{\theta_{i,t}} \left[ \widetilde f_{i,t}(x_t, \theta_{i,t}) \right],
\end{equation}
with
\begin{equation}
\widetilde f_{i,t}(x_t, \theta_{i,t}) \triangleq \frac{1}{|\mathcal{H}_i|} \sum_{j \in \mathcal{H}_i} f_j^{i,t}(x_{i,t}, x_{-i,t}, \theta_{i,t}),
\end{equation}
and the constraint function is given by
\begin{equation}
G_t(x_t) \triangleq \mathbf{E}_{\omega_t} \left[ g_t(x_t, \omega_t) \right].
\end{equation}

Furthermore, the feasible set of problem (\ref{eq.BROMG}) is $\Xi_{t}=\left\{x_{t} \in \Omega \lvert G_{t}\left(x_{t}\right)\leq \bm 0_{m}\right\}$. Let $x^{*}_{i,t}=col(x^{*}_{i1,t},\cdots,x^{*}_{in_{i},t})$ and $x^{*}_{-i,t}=col(x^{*}_{1,t},\cdots,x^{*}_{i-1,t},x^{*}_{i+1,t},\cdots,x^{*}_{N,t})$. The primary goal of this paper is to seek the SGNE of problem (\ref{eq.BROMG}), which is formally defined as follows.
\begin{definition}
\rm
A SGNE at time $t$ is a collective strategy $x_{t}^{*} = \left(x_{i,t}^{*}, x_{-i,t}^{*}\right)$ satisfying that, for all $i \in [N]$,
\begin{equation*}
F_{i,t}\left(x_{i,t}^{*}, x_{-i,t}^{*}\right) \leq \inf_{y \in \Xi_{i,t}(x_{-i,t}^{*})} F_{i,t}\left(y, x_{-i,t}^{*}\right),
\end{equation*}
with
\begin{equation*}
\Xi_{i,t}(x_{-i,t}^{*}) = \left\{x_{i,t} \in \Omega_{i} \mid G_{t}\left(x_{i,t}, x_{-i,t}^{*}\right) \leq \bm{0}_{m} \right\}.
\end{equation*}
\end{definition}

To guarantee the existence of SGNE, the following assumptions are commonly imposed \cite{Cui_2021},\cite{Franci_2022}.

\begin{assumption}\label{assp.4}
\rm{For each $i \in [N]$, $j \in [n_i]$, the set $\Omega_{ij}$ is nonempty, compact and convex, with all elements bounded by a radius $R>0$. Additionally, the feasible set $\Xi_{t}$ satisfies Slater's constraint qualification for any $t \in \left[T\right]$.}
\end{assumption}

\begin{assumption}\label{assp.5}
\rm{For each $i \in [N]$, $j \in \mathcal H_i$, and $t \in [T]$, given any $x_{-i,t} \in \mathbb{R}^{(n-n_i)d}$ and $\omega_t \in \Lambda$, the following conditions hold:
\begin{itemize}
    \item[(i)] The local stochastic cost function $f_{j}^{i,t}(x_{i,t}, x_{-i,t}, \theta_{i,t})$ is convex, Lipschitz continuous, and continuously differentiable with respect to $x_{i,t}$ for each $\theta_{i,t}\in\Theta_i$, where the Lipschitz constant $\ell^{i,t}_{j}(x_{-i,t}, \theta_{i,t})$ is integrable with respect to $\theta_{i,t}$. Moreover, $f_{j}^{i,t}(x_{i,t}, x_{-i,t}, \theta_{i,t})$ is measurable with respect to $\theta_{i,t}$ for each $x_{i,t} \in \mathbb{R}^{n_id}$.

    \item[(ii)] The stochastic constraint function $g_t(x_t, \omega_t)$ is convex and continuously differentiable with respect to $x_{i,t}$.

    \item[(iii)] The functions $F_{i,t}(x_{i,t}, x_{-i,t})$ and $G_t(x_t)$ are continuously differentiable with respect to $x_{i,t}$.
\end{itemize}}
\end{assumption}

Under Assumptions \ref{assp.4}-\ref{assp.5}, the existence of a SGNE for the game \eqref{eq.BROMG} is guaranteed \cite{Ravat_2011}. Among all the possible stochastic generalized Nash equilibria,  we focus on the class of variational stochastic generalized Nash equilibrium (v-SGNE), which constitutes a solution to the suitable stochastic variational inequality (SVI) as follows.
\begin{align*}
\mathbb F_{t}(x^{*}_{t})^{\top}(x_{t}-x^{*}_{t})\geq 0,~~\text{for all}~ x_t\in\Xi_{t},
\end{align*}
where $\mathbb F_{t}(x_{t})\triangleq col\left(\nabla_{x_{ij,t}}F_{i,t}(x_{t})\right)_{i \in [N],j\in[n_{i}]}$ is called the pseudogradient mapping. A standard assumption on the pseudogradient mappings is shown in Assumption \ref{assp.8}.
\begin{assumption}\label{assp.8}
\rm{The pseudogradient mapping $\mathbb F_{t}(x_{t})$ is $\sigma$-strongly monotone with $\sigma>0$ for all $t\in[T]$.}
\end{assumption}

\begin{remark}
\rm{Under Assumption \ref{assp.8}, the SVI admits a unique solution\cite{Ravat_2011}. Consequently, there exists a unique  v-SGNE and the Lagrangian multipliers $\lambda^{*}_{i,t}$ related to the v-SGNE are identical for all $i\in[N]$. Under Assumptions \ref{assp.4}-\ref{assp.8}, the following optimal conditions are satisfied for v-SGNE \cite{Facchinei_2003}, i.e., for any $t\in[T]$,
\begin{equation}\label{re_3}
x_{ij,t}^{*}=P_{\Omega_{ij}}\left[x_{ij,t}^{*}-\gamma_{t}\left(\nabla_{x_{ij,t}}F_{i,t}(x^{*}_{t})+\nabla_{x_{ij,t}}G_{t}(x^{*}_{t})^\top \lambda^{*}_{t}\right)\right],
\end{equation}
\begin{equation}\label{re_2}
\lambda_{t}^{*}=[\lambda_{t}^{*}+G_{t}(x^{*}_{t})]_{+},
\end{equation}}
\end{remark}
where (\ref{re_2}) is equivalent to $\lambda_{t}^{*\top} G_{t}(x^{*}_{t})=0$.

Since random variables exist in the arguments of local cost functions and constraint functions, which follow time-varying probability distributions, the assumptions listed below should be satisfied in order to provide effective theoretical tools for convergence analysis.

\begin{assumption}\label{assp.111} 
\rm{For any bounded closed set $X$, there exist some constants $B_1,B_2 ,B_3,B_4 > 0$ such that for all $i \in [N]$, $j\in\mathcal H_i$, $q \in [n_i]$, $t \in [T]$, and $x_t \in X$, it holds that
\begin{equation*}
\mathbf{E}_{\omega_t} \big\| g_t(x_t, \omega_t) \big\|^2 \leq B_1^2,~\mathbf{E}_{\omega_t} \big\| \nabla_{x_{ij,t}} g_t(x_t, \omega_t) \big\|^2 \leq B_2^2
\end{equation*}
\begin{equation*}
\mathbf{E}_{\theta_{i,t}} \big\| f_j^{i,t}(x_t, \theta_{i,t}) \big\|^2 \leq B_3^2,~\mathbf{E}_{\theta_{i,t}} \big\| \nabla_{x_{iq,t}} f_j^{i,t}(x_t, \theta_{i,t}) \big\|^2 \leq B_4^2.
\end{equation*}}
\end{assumption}

\begin{assumption}\label{assp.7}
\rm{For any $x_{t}=\left(x_{i,t},x_{-i,t}\right)\in \mathbb R^{nd}$ and $y_{t}=\left(y_{i,t},y_{-i,t}\right)\in \mathbb R^{nd}$, there exist some constants $l_{f}>0$, $l_{g_{1}}>0$ and $l_{g_{2}}>0$ such that
\begin{align*}
&\mathbf E_{\theta_{i,t}}\left\|\nabla f_{j}^{i,t}\left(x_{t},\theta_{i,t}\right)-\nabla f_{j}^{i,t}\left(y_{t},\theta_{i,t}\right)\right\|^{2}\leq l_{f}^{2} \left\|x_{t}-y_{t}\right\|^{2},
\\&\mathbf E_{\omega_{t}}\left\|g_{t}\left(x_{t},\omega_{t}\right)-g_{t}\left(y_{t},\omega_{t}\right)\right\|^{2}\leq l_{g_{1}}^{2} \left\|x_{t}-y_{t}\right\|^{2},
\\&\mathbf E_{\omega_{t}}\left\|\nabla g_{t}\left(x_{t},\omega_{t}\right)-\nabla g_{t}\left(y_{t},\omega_{t}\right)\right\|^{2}\leq l_{g_{2}}^{2} \left\|x_{t}-y_{t}\right\|^{2}.
\end{align*}
}
\end{assumption}

\begin{assumption}\label{assp.11}
\rm{The random variables $\theta_{i,t}$ and $\omega_{t}$ are mutually independent for all $i\in[N]$ and $t\in[T]$. Furthermore, the gradients of the local stochastic cost functions satisfy the unbiasedness property, i.e., for any $i\in[N]$, $j\in\mathcal H_i$, $q\in[n_i]$, and $x\in\mathbb R^{nd}$,
\begin{equation*}
\mathbf{E}_{\theta_{i,t}}\left[\nabla_{x_{iq,t}} f_j^{i,t}(x, \theta_{i,t}) \right]=\nabla_{x_{iq,t}}\mathbf{E}_{\theta_{i,t}}\left[f_j^{i,t}(x, \theta_{i,t})\right].
\end{equation*}
\begin{equation*}
\mathbf{E}_{\omega_{t}}\left[\nabla_{x_{iq,t}} g_t(x, \omega_{t}) \right]=\nabla_{x_{iq,t}} \mathbf{E}_{\omega_{t}}\left[g_t(x, \omega_{t}) \right].
\end{equation*}}
\end{assumption}

\subsection{Performance Metrics}
For the purpose of evaluating the performance of the online algorithm, the resilient system-wise regret and constraint violation are defined in the sense of expectation as follows.
\begin{equation}\label{eq.regret}
\mathcal R_{\mathcal H}(T)=\sum_{t=1}^{T}\sum_{i=1}^{N}\sum_{j\in\mathcal H_{i}}\left(F_{i,t}\left(x_{ij,t},x_{-ij,t}^{*}\right)-F_{i,t}\left(x_{t}^{*}\right)\right),
\end{equation}
\begin{equation}\label{eq.cv}
\mathcal{CV}_{\mathcal H}(T)=\sum_{t=1}^{T}\left\|\left[G_{t}\left(x^{*}_{\mathcal H,t}\right)\right]_{+}\right\|,
\end{equation}
where $x_{-ij,t}^{*}$ denotes the SGNE $x^{*}_{t}$ except for $x^{*}_{ij,t}$, i.e.,  $x_{-ij,t}^{*}=col(x^{*}_{11,t},\cdots,x^{*}_{i(j-1),t},x^{*}_{i(j+1),t},\cdots,x^{*}_{Nn_{N},t})$ and $x^{*}_{\mathcal H,t}$ is derived from $x^{*}_{t}$ in which the optimal solution $x_{ij,t}^{*}$ is replaced by $x_{ij,t}$ for any $j\in\mathcal H_{i}$.

\begin{remark}
\rm{The above definitions are totally different from previous works \cite{Meng_2021,Sahoo_2021,Meng_2022,Yu_2023}. In this paper, there are some random variables in the arguments of the cost functions $f^{i,t}_{j}$ and constraint functions $g_{t}$. Meanwhile, Byzantine attacks are considered. Therefore, $\mathcal R_{\mathcal H}(T)$ and $\mathcal{ CV}_{\mathcal H}(T)$ are defined in the sense of expectation and only measure the impact of the decision variables of honest agents on the cost function and constraint function values. Specifically, the definition of $\mathcal R_{\mathcal H}(T)$ considers the performance of all agents from a global perspective, while the definition of $\mathcal{CV}_{\mathcal H}(T)$ only considers the degree of deviation of each honest agent from global constraints, given that other agents have achieved SGNE due to the existence of Byzantine agents. Critically, the SGNE serves as the theoretical benchmark representing the optimal solution of game (\ref{eq.BROMG}) achievable in the absence of Byzantine agents. Measuring (\ref{eq.regret}) and (\ref{eq.cv}) relative to this benchmark quantifies the performance degradation inflicted by malicious attacks on the honest agents.}

The goal is to design the algorithm that enables $\mathcal R_{\mathcal H}(T)$ and $\mathcal{CV}_{\mathcal H}(T)$ to achieve sublinear growth in expectation, i.e., there exist some constants $c_{1},c_{2}\in [0,1)$, such that $\mathbf{E}\left[\mathcal R_{\mathcal H}(T)\right]=\mathcal O(T^{c_{1}})$ and $\mathbf{E}\left[\mathcal{CV}_{\mathcal H}(T)\right]=\mathcal O(T^{c_{2}})$.
\end{remark}
\section{Algorithm Development}\label{sec.two}
This section introduces the mod-SVRG technique to reduce the variances of randomness and endows each honest agent with auxiliary variables to estimate other agents' decisions. Additionally, a robust aggregation mechanism to mitigate the impact of malicious attacks on honest agents is introduced. Subsequently, a dynamic average consensus algorithm is constructed to track global information and an online SGNE seeking algorithm is designed ultimately.

A regularized time-varying stochastic Lagrangian function related to (\ref{eq.BROMG}) for cluster $i$ is written as
\begin{equation*}
\hat {\mathcal L}_{i,t}(x_{t},\lambda_{t},\theta_{i,t},\omega_{t})=\widetilde f_{i,t}\left(x_{t},\theta_{i,t}\right)+\lambda_{t} g_{t}\left(x_{t},\omega_{t}\right)-\frac{\beta_{t}}{2}\left\|\lambda_{t}\right\|^{2},
\end{equation*}
where $\lambda_{t}\in \mathbb R^{m}_{\geq 0}$ is the dual variable and $\beta_{t}$ is the non-negative regularization parameter.

For ease of notations, denote $\nabla_{x_{ij,t}}\hat {\mathcal L}_{i,t}(x_{t},\lambda_{t},\theta_{i,t},\omega_{t})$ and $\nabla_{\lambda_{t}}\hat {\mathcal L}_{i,t}(x_{t},\lambda_{t},\theta_{i,t},\omega_{t})$ as $M_{ij,t}$ and $N_{t}$, repectively. Then $M_{ij,t}$ and $N_{t}$ can be represented as
\begin{equation}\label{eq.gradient}
\begin{aligned}
&M_{ij,t}=\nabla_{x_{ij,t}}\widetilde f_{i,t}\left(x_{t},\theta_{i,t}\right)+\lambda_{t}\nabla_{x_{ij,t}}g_{t}\left(x_{t},\omega_{t}\right),
\\&N_{t}=g_{t}\left(x_{t},\omega_{t}\right)-\beta_{t}\lambda_{t}.
\end{aligned}
\end{equation}

Intuitively speaking, $M_{ij,t}$ and $N_{t}$ are the gradient descent direction for the decision variable $x_{ij,t}$ and the gradient ascent direction for the dual variable $\lambda_{t}$, respectively. It is highly likely that each agent cannot distinguish the malicious messages from the stochastic gradients \cite{Wu_2020}. To be specific, if agent $j$ in cluster $i$ randomly selects a sample $\theta_{i,t}\sim\mathcal P_{i,t}$ such that $\nabla_{x_{ij,t}}f_{j}^{i,t}\left(x_{t},\theta_{i,t}\right)$ deviates from the mean, then this agent is likely to be  misidentified as a Byzantine agent and its information may be filtered out. To avoid this phenomenon, the mod-SVRG technique is proposed to reduce such variances.

\subsection{Variance Reduction Technique}
Before going on, one can see that $M_{ij,t}$ and $N_{t}$ are the unbiased estimators of $\nabla_{x_{ij,t}}F_{i,t}(x_{t})+\lambda_{t}\nabla_{x_{ij,t}}G_{t}\left(x_{t}\right)$ and $G_{t}\left(x_{t}\right)-\beta_{t}\lambda_{t}$ conditioned on $\mathcal U_{t}$, respectively, where $\mathcal U_{t}$ is the $\sigma$-field  generated by the random variables $\theta_{1,t},\cdots,\theta_{N,t},\omega_{t}$. However, both $M_{ij,t}$ and $N_{t}$ can be affected by randomness, resulting in a likely large variance. Therefore, the mod-SVRG technique (see Appendix \ref{mod-SVRG}) is employed to reduce the variance and consequently construct the gradient descent direction $\widetilde M_{ij,t}$ for $x_{ij,t}$ and the gradient ascent direction  $\widetilde N_{t}$ for $\lambda_{t}$, as detailed below.
\begin{equation}\label{direction}
\begin{aligned}
&\widetilde M_{ij,t}=d_{ij,t}^{1}+\lambda_{t}d_{ij,t}^{3},
&\widetilde N_{t}=d_{t}^{2}-\beta_{t}\lambda_{t},
\end{aligned}
\end{equation}
where $d_{ij,t}^{1}$ and $d_{t}^{2}$ satisfy
\begin{equation}\label{firstpart}
\begin{aligned}
&d_{ij,t}^{1}=\nabla_{x_{ij,t}}\widetilde f_{i,t}\left(x_{t},\theta_{i,t}\right)
-\nabla_{x_{ij,t}}\widetilde f_{i,t}\left(\tau_{t},\theta_{i,t}\right)
\\&\quad\quad\quad+\nabla_{x_{ij,t}}F_{i,t}\left(\tau_{t}\right),
\\&d_{t}^{2}=g_{t}\left(x_{t},\omega_{t}\right)-g_{t}\left(\tau_{t},\omega_{t}\right)+G_{t}(\tau_{t}),
\end{aligned}
\end{equation}
and $d_{ij,t}^{3}$ is the gradient of $d_{t}^{2}$ at $x_{ij,t}$, meaning that
\begin{equation}\label{secondpart}
\begin{aligned}
d_{ij,t}^{3}=\nabla_{x_{ij,t}}g_{t}\left(x_{t},\omega_{t}\right)-\nabla_{x_{ij,t}}g_{t}\left(\tau_{t},\omega_{t}\right)+\nabla_{x_{ij,t}}G_{t}(\tau_{t}).
\end{aligned}
\end{equation}

Therein, given a time-varying positive integer $s(t)$ satisfying $\lim\limits_{t\to\infty}s(t)=1$, $\tau_{t}$ is set as $\tau_{t}=x_{t}$ if $\bmod(t,s(t))=0$ and $\tau_{t}=\tau_{t-1}$ otherwise. Apparently, $d_{ij,t}^{1}$, $d_{t}^{2}$, and $d_{ij,t}^{3}$ serve as unbiased estimators for $\nabla_{x_{ij,t}}F_{i,t}(x_{t})$, $G_{t}(x_{t})$, and $\nabla_{x_{ij,t}}G_{t}(x_{t})$, respectively. Thus, $\widetilde M_{ij,t}$ and $\widetilde N_{t}$ are also unbiased estimators of $\nabla_{x_{ij,t}}F_{i,t}(x_{t})+\lambda_{t}\nabla_{x_{ij,t}}G_{t}\left(x_{t}\right)$ and $G_{t}\left(x_{t}\right)-\beta_{t}\lambda_{t}$ conditioned on $\mathcal U_{t}$, respectively. Furthermore, the variances of $d_{ij,t}^{1}$, $d_{t}^{2}$ and $d_{ij,t}^{3}$ will converge to $0$.  Below is a proof taking $d_{ij,t}^{1}$ as an example. By Assumptions \ref{assp.7} and \ref{assp.11}, and the fact that $\mathbf E\left\|x-\mathbf Ex\right\|^{2}=\mathbf E\left\|x\right\|^{2}-\left\|\mathbf E x\right\|^{2}$, one has
\begin{align*}
&\lim\limits_{t\to\infty}\mathbf E_{\mathcal U_{t}}\left\|d_{ij,t}^{1}-\nabla_{x_{ij,t}}F_{i,t}(x_{t})\right\|^{2}
\\&=\lim\limits_{t\to\infty}\mathbf E_{\mathcal U_{t}}\left\|\nabla_{x_{ij,t}}\widetilde f_{i,t}\left(x_{t},\theta_{i,t}\right)-\nabla_{x_{ij,t}}\widetilde f_{i,t}\left(\tau_{t},\theta_{i,t}\right)\right.
\\&\quad\quad\quad\quad\quad\left.+\nabla_{x_{ij,t}}F_{i,t}\left(\tau_{t}\right)-\nabla_{x_{ij,t}}F_{i,t}(x_{t})\right\|^{2}
\\&\leq l_{f}^{2}\lim\limits_{t\to\infty}\left\|x_{t}-\tau_{t}\right\|^{2},
\end{align*}
where $\lim\limits_{t\to\infty}\mathbf E_{\mathcal U_{t}}\left\|d_{ij,t}^{1}-\nabla_{x_{ij,t}}F_{i,t}(x_{t})\right\|^{2}\geq 0$ holds naturally due to the non-negativity of norms, and $\lim\limits_{t\to\infty}\left\|x_{t}-\tau_{t}\right\|^{2}=0$ holds given the condition that $\lim\limits_{t\to\infty}s(t)=1$. Therefore, $\lim\limits_{t\to\infty}\mathbf E_{\mathcal U_{t}}\left\|d_{ij,t}^{1}-\nabla_{x_{ij,t}}F_{i,t}(x_{t})\right\|^{2}=0$ by using the Sandwich Theorem. The proofs related to $d_{t}^{2}$ and $d_{ij,t}^{3}$ are similar to that of $d_{ij,t}^{1}$.

\begin{remark}
\rm{The design inspiration of mod-SVRG is derived from the SVRG technique, though SVRG is specifically designed for solving finite-sum optimization problems. Nevertheless, the problem \(\underset{x}{\rm{min}} \frac{1}{n} \sum_{i=1}^{n} f_i(x)\) can be regarded as a special case of the general stochastic optimization problem \(\underset{x}{\rm{min}}\, \mathbb{E}_{\xi}[f(x, \xi)]\), where the random variable \(\xi\) admits finite samples, and the sample size \(n\) may be large or even infinite. The mod-SVRG technique is firstly proposed for application in online settings and the general stochastic optimization problems. According to the analysis in Appendix \ref{mod-SVRG}, the mod-SVRG technique and the SVRG technique share the same geometric convergence rate in expectation. For ease of analysis, it is assumed that the gradients $\nabla_{x_{ij,t}}F_{i,t}\left(\tau_{t}\right)$ and $\nabla_{x_{ij,t}}G_{t}\left(\tau_{t}\right)$ and the constraint function $G_{t}\left(\tau_{t}\right)$ can be computed exactly.}
\end{remark}

Furthermore, the functions $f_{j}^{i,t}\left(x_{t},\theta_{i,t}\right)$ and $g_{t}\left(x_{t},\omega_{t}\right)$ also depend on the decision variables of other agents for each agent $j$ in cluster $i$. If one of the agents is malicious, then the game progress can be arbitrarily controlled. In what follows, each agent is equipped with an auxiliary variable to estimate other agents' decisions.

\subsection{Decoupling Decision Variables}
For each agent $j \in \mathcal H_{i}$, $x_{t}=(x_{i,t},x_{-i,t})$ in the arguments of functions $f^{i,t}_{j}(x_{t},\theta_{i,t})$ and $g_{t}(x_{t},\omega_{t})$ may include the decisions of Byzantine agents, which is the main difference between the Byzantine games and Byzantine optimization. To overcome the above dilemma and the lack of knowledge of other agents' decisions in decentralized settings, each honest agent maintains an estimate of other agents' decisions. 

Specifically, each agent $j \in \mathcal H_{i}$ keeps a vector $\bm x_{ij,t}=col(\bm x_{ij,t}^{1},\cdots,\bm x_{ij,t}^{N})\triangleq col(\bm x_{ij,t}^{p})_{p\in[N]} \in \mathbb R^{nd}$, where $\bm x_{ij,t}^{p} \in \mathbb R^{n_{p}d}$ represents the honest agent $j$ in cluster $i$'s estimation of the agents in cluster $p$ at time $t$. Furthermore, $\bm x_{ij,t}^{p}=col(\bm x_{ij,t}^{p1},\cdots,\bm x_{ij,t}^{pn_{p}})\triangleq col(\bm x_{ij,t}^{pq})_{q\in[n_{p}]}$ with $\bm x_{ij,t}^{pq}\in\mathbb R^{d}$ representing the honest agent $j$ in cluster $i$'s estimation of agent $q$ in cluster $p$ at time $t$. Clearly, $\bm x_{ij,t}^{ij}=x_{ij,t}$. An effective estimate is that $\bm x_{ij,t}$ of all honest agents reach consensus as time $t$ approaches infinity. Based on these advantages, $x_{t}$ is substituted by $\bm x_{ij,t}$ to ensure that the functions only depend on its own estimated variables.

Each honest agent must interact with its neighbors to achieve an agreement on $\bm x_{ij,t}$ as much as possible. However, Byzantine agents may send misleading information to their neighbors. Let $\widetilde{\bm x}_{rh,t}^{pq}$ denote the message that agent $h$ in cluster $r$ sends to its neighbors at time $t$. Precisely,
\begin{equation}\label{send value_three}
\widetilde{\bm x}_{rh,t}^{pq}=\left\{\begin{aligned}&\bm x_{rh,t}^{pq},~~~&h\in\mathcal H_{r},\\&*,~~&h\in\mathcal B_{r}.\end{aligned}\right.
\end{equation}
where $*$ represents an arbitrary $d$-dimensional vector.

To address this issue, one of the robust aggregation mechanisms \rm{is adopted to aggregate the received messages, called the} \it{coordinate-wise trimmed mean (CTM)} \rm{method, which is robust to Byzantine attacks \cite{Wang_2022}. To be more specific, each agent $j\in\mathcal H_{i}$ sorts the received values $[\widetilde{\bm x}_{rh,t}^{pq}]_{k}$ from incoming neighbor $h$ in cluster $r$, then it rejects the top $b$ and the bottom $b$ values for $\forall k \in [d]$ and iteratively updates using the average of the remaining input values. That is, the outliers are discarded iteratively. The above filtering step is summarized as an algorithm function $Trim\left\{\cdot\right\}$ for ease of exposition.

\subsection{Dynamic Average Consensus}
It can be observed that $d_{ij,t}^{1}$ is composed of the cost functions of all honest agents in cluster $i$. However, each agent $j\in\mathcal H_{i}$ only has access to its own cost functions $f_{j}^{i,t}$ instead of the cost functions of all honest agents in cluster $i$. Thus, $v_{ij,t}^{j}$ is constructed by $j \in \mathcal H_{i}$ to estimate the global direction $d_{ij,t}^{1}$ and follows the dynamic average consensus step inspired by \cite{Zhu_2010}, i.e.,
\begin{equation}\label{update_presence_1}
v_{iq,t}^{j}=Trim\left\{(\widetilde {v}^{h}_{iq,t-1})_{h \in \mathcal N^{i}_{ij}}\right\}+\epsilon_{iq,t}^{j}-\epsilon_{iq,t-1}^{j},\forall q\in\mathcal H_{i},
\end{equation}
where $\epsilon_{iq,t}^{j}=\nabla_{x_{iq,t}}f_{j}^{i,t}\left(\bm x_{ij,t},\theta_{i,t}\right)-\nabla_{x_{iq,t}}f_{j}^{i,t}\left(\bm\tau_{ij,t},\theta_{i,t}\right)+\nabla_{x_{iq,t}}\mathbf E_{\theta_{i,t}}\left[f_{j}^{i,t}\left(\bm\tau_{ij,t},\theta_{i,t}\right)\right]$, $\bm\tau_{ij,t}=\bm x_{ij,t}$ if $\bmod(t,s(t))=0$ and $\bm\tau_{ij,t}=\bm \tau_{ij,t-1}$ otherwise, and $\widetilde v_{iq,t}^{h}$ represents the message that agent $h\in[n_{i}]$ sends to its neighbors in cluster $i$ at time $t$, given by
\begin{equation}\label{send value_one}
\widetilde v_{iq,t}^{h}=\left\{\begin{aligned}&v^{h}_{iq,t},~~~&h \in\mathcal H_{i},\\&*,~~&h \in \mathcal B_{i}.\end{aligned}\right.
\end{equation}

In summary, the decentralized Byzantine-resilient online SGNE seeking algorithm is designed as Algorithm 1.

\begin{algorithm}\label{Algorithm}
\caption{Decentralized Byzantine-resilient Online SGNE Seeking Algorithm (DBROSA) from the view point of $j\in\mathcal H_{i}$}
\textbf{Input:} The step sizes $\left\{\alpha_{t}\right\}_{t=1}^{T}$, $\left\{\beta_{t}\right\}_{t=1}^{T}$, $\left\{\gamma_{t}\right\}_{t=1}^{T}$ and $\left\{\eta_{t}\right\}_{t=1}^{T}$, the tunable parameters $\left\{\delta_{t}\right\}_{t=1}^{T}$ and $\left\{\zeta_{t}\right\}_{t=1}^{T}$, the upper bound $b$ and $b_{i}$, and the time-varying positive integer $s(t)$.

\textbf{Initialize:} $x_{ij,1}\in \Omega_{ij}$, $G_{1}(x^{*}_{\mathcal H,1})\leq \bm 0_{m}$, $\bm x_{ij,0}=\bm x_{ij,1}\in \mathbb R^{nd}$, $\lambda_{ij,1}=\bm 0_{m}$, $\theta_{i,1}\sim \mathcal P_{i,1}$, $\omega_{1}\sim \mathcal Q_{1}$. For any $q\in\mathcal H_{i}$, $v_{iq,1}^{j}=\epsilon_{iq,1}^{j}=\bm 0_{d}$.

\textbf{for} $t=1,2,\cdots,T$ \textbf{do}
\begin{enumerate}
\item[] \textbf{Local filtering step:} 
\item[]\quad For $\forall k \in [d]$, each honest agent $j$ in cluster $i$ sorts the values $[\widetilde v^{h}_{iq,t}]_{k}$ received from its incoming neighbors $h$ in cluster $i$ over graph $\mathcal G_{i}$. Furthermore, denote $\mathcal U_{ij,t}^{k}$ as the set of agents that sent the top $b_i$ values and denote $\mathcal L_{ij,t}^{k}$ as the set of agents that sent the bottom $b_i$ values.  Let $\mathcal R_{ij,t}^{k}=(\mathcal N_{ij}^{i}\setminus (\mathcal U_{ij,t}^{k} \cup \mathcal L_{ij,t}^{k}))\cup\left\{j\right\}$.
\item[] \textbf{Global filtering step:} 
\item[]\quad For $\forall k \in [d]$, each honest agent $j$ in cluster $i$ sorts the values $[\widetilde{\bm x}_{rh,t}^{pq}]_{k}$ received from its neighbors $h$ in cluster $r$ over graph $\mathcal G$. Furthermore, denote $\mathcal U_{ij,t}^{pqk}$ as the set of agents that sent the top $b$ values and denote $\mathcal L_{ij,t}^{pqk}$ as the set of agents that sent the bottom $b$ values.  Let $\mathcal Y_{ij,t}^{pqk}=(\mathcal N_{ij}\setminus (\mathcal U_{ij,t}^{pqk} \cup \mathcal L_{ij,t}^{pqk}))\cup\left\{(i,j)\right\}$.
\item[] \textbf{Updating step:} 
\item[]\quad Each honest agent $j$ in cluster $i$ receives the information of $f^{i,t}_{j}$, $g_{t}$, $\mathcal P_{i,t}$, and $\mathcal Q_{t}$.
\item[]\quad Each honest agent $j$ in cluster $i$ selects $\theta_{i,t}\sim \mathcal P_{i,t}$, $\omega_{t}\sim \mathcal Q_{t}$, and computes $\bm \tau_{ij,t}$ following the rules that $\bm\tau_{ij,t}=\bm x_{ij,t}$ if $mod(t,s(t))=0$ and $\bm\tau_{ij,t}=\bm \tau_{ij,t-1}$ otherwise. 
\item[]\quad The update rules are shown in (\ref{eq.1})-(\ref{eq.5}).
\end{enumerate}
\textbf{end for}
\end{algorithm}

To solve (\ref{eq.BROMG}), at time $t$, every honest agent $j$ in cluster $i$ updates its decision variable $x_{ij,t}$ and dual variable $\lambda_{ij,t}$ as follows.
\begin{equation}\label{eq.1}
y_{ij,t}=g_{t}\left(\bm x_{ij,t},\omega_{t}\right)-g_{t}\left(\bm \tau_{ij,t},\omega_{t}\right)+G_{t}(\bm\tau_{ij,t}),
\end{equation}
\begin{equation} \label{eq.2}
\lambda_{ij,t+1}=\left[\lambda_{ij,t}+\eta_{t}(y_{ij,t}-\beta_{t}\lambda_{ij,t})\right]_{+},
\end{equation}
\begin{equation}\label{eq.3}
\begin{aligned}
x_{ij,t+1}=&(1-\alpha_{t})x_{ij,t}
\\&+\alpha_{t}P_{\Omega_{ij}}\bigg[x_{ij,t}-\gamma_{t}v_{ij,t}^{j}-\gamma_{t}(\nabla_{x_{ij,t}}y_{ij,t})^\top\lambda_{ij,t+1}\bigg],
\end{aligned}
\end{equation}
\begin{equation}\label{eq.4}
\begin{aligned}
&[\bm x_{ij,t+1}^{pq}]_{k}=\mathbb{I}\left((i,j)\neq (p,q)\right)\bigg(\delta_{t}\frac{\sum_{(r,h) \in \mathcal Y_{ij,t}^{pqk}}[\widetilde{\bm x}_{rh,t}^{pq}]_{k}}{|\mathcal N_{ij}|-2b+1}
\\&\quad\quad\quad\quad\quad+(\zeta_{t}-\delta_{t})[\bm x_{ij,t}^{pq}]_{k}\bigg)+\mathbb{I}\left((i,j)= (p,q)\right)[x_{ij,t+1}]_{k},
\end{aligned}
\end{equation}
\begin{equation}\label{eq.5}
[v_{iq,t+1}^{j}]_{k}=\frac{\sum_{h\in\mathcal R_{ij,t}^{k}}[\widetilde v_{iq,t}^{h}]_{k}}{|\mathcal N_{ij}^{i}|-2b_{i}+1}+[\Delta \epsilon_{iq,t+1}^{j}]_{k},\forall q\in\mathcal H_{i},
\end{equation}
where $\alpha_{t}$, $\beta_{t}$, $\gamma_{t}$, and $\eta_{t}$ are the time-varying step sizes, $\delta_{t}$ and $\zeta_{t}$ are the tunable parameters satisfying $0<\delta_{t}<\zeta_{t}<1$ and $0<|\mathcal H|\delta_{t}+\zeta_{t}<1$, and $\Delta \epsilon_{iq,t+1}^{j}=\epsilon_{iq,t+1}^{j}-\epsilon_{iq,t}^{j}$.

Observe that Algorithm 1 is fully decentralized. Malicious or misleading messages will be received by honest agents due to the unknown identity of Byzantine agents. Therefore, a robust aggregation mechanism is used in (\ref{eq.4})-(\ref{eq.5}) to filter out the collected information $\widetilde{\bm x}_{rh,t}^{pq}$ and $\widetilde v_{iq,t}^{h}$. Let $\mathscr U_{t}$ denote the $\sigma$-algebra generated by $\cup_{s=1}^{t}\mathcal U_{s}$. By Assumption \ref{assp.11}, it is evident that $y_{ij,t-1},\lambda_{ij,t},x_{ij,t},\bm x^{pq}_{ij,t}$ and $v^{j}_{iq,t-1}$ depend on $\mathscr U_{t-1}$ and are independent of $\mathcal U_{s}$ for all $s\geq t$.

\section{Main Results}\label{sec.three}
The goal of this section is to prove that the resilient system-wise regret and constraint violation constructed by (\ref{eq.regret}) and (\ref{eq.cv}) grow sublinearly.

To depict the fluctuation of the time-varying SGNE $x_{t}^{*}$, $\Phi(T)$ is defined as the summation of the differences between the decisions $x^{*}_{ij,t}$ at the current time step $t$ and those at the previous time step as follows.
\begin{equation}\label{Phi(T)}
\Phi(T)=\Phi_{\mathcal H}(T)+\Phi_{\mathcal B}(T),
\end{equation}
where 
\begin{align*}
\Phi_{\mathcal H}(T)=\sum_{t=1}^{T}\sum_{i\in[N]}\sum_{j\in\mathcal H_{i}}\sum_{k=1}^{d}\left\|[x_{ij,t+1}^{*}]_{k}-[x_{ij,t}^{*}]_{k}\right\|,
\end{align*}
and
\begin{align*}
\Phi_{\mathcal B}(T)=\sum_{t=1}^{T}\sum_{i\in[N]}\sum_{h\in\mathcal B_{i}}\sum_{k=1}^{d}\left\|[x_{ih,t+1}^{*}]_{k}-[x_{ih,t}^{*}]_{k}\right\|.
\end{align*}

\begin{lemma}\label{lemma_21}
\rm{If Assumption \ref{assp.2} holds, then for any $i,p \in [N]$ and $k \in [d]$, the update rules (\ref{eq.4})-(\ref{eq.5}) are equivalent to
\begin{equation}\label{row_form_4}
\begin{aligned}
\bm x^{pqk}_{t+1}&=\delta_{t}\bm M^{1k-}_{pq,t}\bm x^{pqk}_{t}+(\zeta_{t}-\delta_{t})\bm x^{pqk}_{t}
\\&\quad+\delta_{t}[x_{pq,t}]_{k}[\bm M^{1k-}_{pq,t}]_{(:,\mathcal S(p,q))},~~\forall q\in\mathcal H_{p},
\end{aligned}
\end{equation}
\begin{equation}
\bm x^{phk+}_{t+1}=\delta_{t}\bm M^{2k}_{ph,t}\bm x^{phk+}_{t}+(\zeta_{t}-\delta_{t})\bm x^{phk+}_{t},~~\forall h\in\mathcal B_{p},
\end{equation}
\begin{equation}\label{row_form_1}
\bm v_{iq,t+1}^{k}=\bm Y_{iq,t}^{k}\bm v_{iq,t}^{k}+\bm \epsilon_{iq,t+1}^{k}-\bm \epsilon_{iq,t}^{k},~\forall q\in\mathcal H_{i},
\end{equation}
where $\bm x^{pqk}_{t}=col([\bm x^{pq}_{ij,t}]_{k})_{i\in[N], j\in\mathcal H_{i},(i,j)\neq(p,q)}$, $\bm x^{phk+}_{t}=col([\bm x^{ph}_{ij,t}]_{k})_{i\in[N], j\in\mathcal H_{i}}$, $\bm v_{iq,t}^{k}=col([v_{iq,t}^{j}]_{k})_{j\in\mathcal H_{i}}$, $\bm \epsilon_{iq,t}^{k}=col([\epsilon_{iq,t}^{j}]_{k})_{j \in \mathcal H_{i}}$, $\bm M_{pq,t}^{1k-}$ is a time-varying row stochastic matrix $\bm M_{pq,t}^{1k}$ that remains after removing the $\mathcal S(p,q)$-th row and $\mathcal S(p,q)$-th column, $\bm M_{pq,t}^{2k}$ is a time-varying row stochastic matrix. Specifically, $\mathcal S(p,q)\triangleq\sum_{h=1}^{p-1}|\mathcal H_{h}|+q$ and $\sum_{h=1}^{0}|\mathcal H_{h}|\triangleq 0$. Furthermore, $\bm Y_{iq,t}^{k}$ is a time-varying row stochastic matrix depending on the behavior of Byzantine agents and $\bm v_{iq,t}^{k}$, and is constructed following the way in \cite{Wu_2010}.}
\end{lemma}
{\bf Proof.} The proof follows the same reasoning as that of Claim 2 in \cite{Vaidya_2012}.\hfill$\blacksquare$

\begin{remark}
\rm{From Lemma \ref{lemma_21}, Assumption \ref{assp.2} is essential for two key reasons. Firstly, it ensures that the local variables (i.e., $[\bm{x}^{pq}_{ij,t}]_k$ and $[v^{j}_{iq,t}]_k$) of each honest agent can be represented as linear combinations of the corresponding variables from its honest neighbors and its own local variables. Secondly, it guarantees that the CTM method remains well-defined, as  each honest agent discards $2b_i$ (or $2b$) of the largest and smallest values received from neighbors, thereby preserving at least one value from neighbors.}
\end{remark}

Based on (\ref{row_form_4}), the following lemma establishes an upper bound of the estimation error for each honest agent.
\begin{lemma}\label{lemma_6}\label{lemma_2}
\rm{Under Assumption \ref{assp.2}, for any $i,p \in [N]$, $j \in \mathcal H_{i}$, $h \in \mathcal B_{p}$, $q\in \mathcal H_{p}$, and $t\geq2$, there exists a constant $v\in(0,1)$ such that
\begin{equation}\label{eq.01}
\begin{aligned}
\left\|\bm x^{pq}_{ij,t}-x_{pq,t}\right\|&\leq C_{1}v^{t-1}+C_{2}\sum_{r=0}^{t-2}v^{r}(\alpha_{t-r-1}+\xi_{t-r-1}),
\end{aligned}
\end{equation}
\begin{equation}\label{eq.02}
\begin{aligned}
\left\|\bm x^{ph}_{ij,t}-x_{ph,t}^{*}\right\|&\leq \widetilde C_{1}v^{t-1}+C_{3}\sum_{r=0}^{t-2}v^{r}\Phi_{\mathcal B}^{ph}(t-r-1)\quad\quad\quad
\\&\quad+C_{4}\sum_{r=0}^{t-2}v^{r}\xi_{t-r-1},
\end{aligned}
\end{equation}
where $C_{1}=\sum_{k=1}^{d}\sum_{i\in[N]}\sum_{j\in \mathcal H_{i}}\left\|[\bm x^{pq}_{ij,1}]_{k}-[x_{pq,1}]_{k}\right\|$, $\widetilde C_{1}=\sum_{k=1}^{d}\sum_{i\in[N]}\sum_{j\in \mathcal H_{i}}\left\|[\bm x^{ph}_{ij,1}]_{k}-[x_{ph,1}^{*}]_{k}\right\|$, $C_{2}=2d\sqrt{|\mathcal H|}R$, $C_{3}=\sqrt{|\mathcal H|}$, $C_{4}=d\sqrt{|\mathcal H|}R$, $\Phi_{\mathcal B}^{ph}(t-r-1)=\sum_{k=1}^{d}\left\|[x^{*}_{ph,t-r}]_{k}-[x^{*}_{ph,t-r-1}]_{k}\right\|$, and $\xi_{t-r-1}=1-\zeta_{t-r-1}$.}
\end{lemma}
{\bf Proof.} The proof is deferred to Appendix \ref{prof_lemma_2}.\hfill$\blacksquare$

Define $G_{iqj}^t(x) \triangleq \nabla_{x_{iq,t}} \mathbf{E}_{\theta_{i,t}} \left[ f_j^{i,t}(x, \theta_{i,t}) \right]$, then
\begin{align*}
\Delta F^{\text{sup}}_T
\triangleq \sup_{x \in \Omega} 
\sum_{t=2}^T \sum_{i \in [N]} \sum_{j, q \in \mathcal H_i}
\left\| G_{iqj}^t(x) - G_{iqj}^{t-1}(x) \right\|.
\end{align*}

An upper bound of the tracking error between $v^{j}_{iq,t}$ and $\frac{1}{|\mathcal H_{i}|}\sum_{j\in\mathcal H_{i}}\epsilon^{j}_{iq,t}$ is given by Lemma \ref{lemma_3}.
\begin{lemma}\label{lemma_3}
\rm{Under Assumptions \ref{assp.2}-\ref{assp.5},\ref{assp.7}, if $\min \limits_{q\in\mathcal H_{i}}|\mathcal N^{i}_{iq}|>\frac{|\mathcal H_{i}|}{2}+2b_i-1$, then for any $i \in [N]$ and $q \in \mathcal H_{i}$, there holds
\begin{equation}\label{eq.03}
\begin{aligned}
&\sum_{t=1}^{T}\mathbf E_{\mathcal U_{t}}\left\|\bm v_{iq,t}-(\bm 1_{|\mathcal H_{i}|}\otimes I_{d})\bar \epsilon_{iq,t}\right\|\\&\leq \mathcal O\left(1+\sum_{t=1}^{T}\alpha_{t}+\sum_{t=1}^{T}\xi_{t}+\Phi(T)+\Delta F^{\text{sup}}_{T}\right),
\end{aligned}
\end{equation}
where $\bm v_{iq,t}=col(v^{j}_{iq,t})_{j \in \mathcal H_{i}}$ and $\bar \epsilon_{iq,t}=\frac{1}{|\mathcal H_{i}|}\sum_{j\in\mathcal H_{i}}\epsilon^{j}_{iq,t}$.}
\end{lemma}
{\bf Proof.} The proof is deferred to Appendix \ref{prof_lemma_3}.\hfill$\blacksquare$

The following theorems establish upper bounds for both $\mathcal{R}_{\mathcal{H}}(T)$ and $\mathcal{CV}_{\mathcal{H}}(T)$.
\begin{theorem}\label{Theorem_1}
\rm{Under Assumptions \ref{assp.2}-\ref{assp.11}, if $\left\{\beta_{t}\right\}_{t\in[T]}$ and $\left\{\eta_{t}\right\}_{t\in[T]}$ are constant step sizes and $0<\beta_{t}\eta_{t}<\frac{1}{2}$, then
\begin{equation}\label{eq.05}
\begin{aligned}
\mathbf{E}\left[\mathcal{CV}_{\mathcal H}(T)\right]&\leq \mathcal O\Big(T^{\frac{1}{2}}+T^{\frac{1}{2}}\sum_{t=1}^{T}\alpha_{t}^{\frac{1}{2}}+T^{\frac{1}{2}}\sum_{t=1}^{T}\xi_{t}^{\frac{1}{2}}+T^{\frac{1}{2}}\sqrt{\Phi(T)}
\\&\quad\quad\quad+\beta_{T}^{\frac{1}{2}}\eta_{T}^{\frac{1}{2}}T^{\frac{3}{2}}\Big).
\end{aligned}
\end{equation}}
\end{theorem}
{\bf Proof.} The proof is deferred to Appendix \ref{prof_lemma_7}.\hfill$\blacksquare$

\begin{theorem}\label{Theorem_2}
\rm{Under Assumptions \ref{assp.2}-\ref{assp.11}, if there exists $t_{1}\in[T]$, such that $\alpha_{t_{1}}\gamma_{t_{1}}\leq \alpha_{t}\gamma_{t}$ for any $t\in[1,t_{1}]$, $\alpha_{t}$, $\gamma_{t}$ are time-invariant for any $t>t_{1}$ and $\alpha_{t_{1}}\gamma_{t_{1}}\leq \alpha_{t_{1}+1}\gamma_{t_{1}+1}$, $\left\{\beta_{t}\right\}_{t\in[T]}$ and $\left\{\eta_{t}\right\}_{t\in[T]}$ are constant step sizes, $0<\beta_{t}\eta_{t}<\frac{1}{2}$, and $\min \limits_{q\in\mathcal H_{i}}|\mathcal N^{i}_{iq}|>\frac{|\mathcal H_{i}|}{2}+2b_i-1$, then
\begin{equation}\label{eq.06}
\begin{aligned}
&\mathbf{E}\left[\mathcal R_{\mathcal H}(T)\right]\leq\mathcal O\Big(T^{\frac{1}{2}}+T^{\frac{1}{2}}\sum_{t=1}^{T}\alpha_{t}^{\frac{1}{2}}+T^{\frac{1}{2}}\sum_{t=1}^{T}\xi_{t}^{\frac{1}{2}}+T^{\frac{1}{2}}\sqrt{\Phi(T)}
\\&\quad+T\sum_{t=1}^{T}\gamma_{t}^{\frac{1}{2}}+T^{\frac{3}{2}}\eta_{T}^{\frac{1}{2}}+T^{\frac{3}{2}}\eta_{T}\sum_{t=1}^{T}\gamma_{t}^{\frac{1}{2}}+T\Delta F^{\text{sup}}_{T}\sum_{t=1}^{T}\gamma_{t}^{\frac{1}{2}}
\\&\quad+T^{\frac{1}{2}}\sqrt{\Delta F^{\text{sup}}_{T}}+T^{\frac{1}{2}}\sqrt{\mathcal{CV}_{\mathcal H}(T)}\Big).~~~~~~~~~~~~~~
\end{aligned}
\end{equation}}
\end{theorem}
{\bf Proof.} The proof is deferred to Appendix \ref{prof_lemma_8}.\hfill$\blacksquare$

\begin{assumption}\label{assp.12}
\rm{The accumulated variation $\Phi(T)$ and $\Delta F^{\text{sup}}_{T}$ grow sublinearly, i.e., $\Phi(T)=\mathcal O(T^{\phi})$ and $\Delta F^{\text{sup}}_{T}=\mathcal O(T^{s})$ with $\phi,s \in [0,1)$.}
\end{assumption}

\begin{assumption}\label{assp.13}
\rm{The step sizes $\left\{\alpha_{t}\right\}_{t \in \left[T\right]}$, $\left\{\beta_{t}\right\}_{t \in \left[T\right]}$ and $\left\{\gamma_{t}\right\}_{t \in \left[T\right]}$, and the tunable parameters $\left\{\zeta_{t}\right\}_{t \in \left[T\right]}$ and $\left\{\eta_{t}\right\}_{t \in \left[T\right]}$ exhibit the properties listed below.}
\begin{enumerate}
\item{$\left\{\eta_{t}\right\}_{t \in \left[T\right]}$ is a constant step size satisfying $\eta_{t}=T^{\eta}$ with $\eta<-1$},
\item{$\left\{\alpha_{t}\right\}_{t \in \left[T\right]}$ and $\left\{\gamma_{t}\right\}_{t \in \left[T\right]}$ are segmented by limited time $t_{1}$, i.e.,\begin{equation*}\alpha_{t}=\left\{\begin{aligned}&t^{a_{1}},~~~&t\leq t_{1},\\&T^{a_{2}},~~&t>t_{1}.\end{aligned}\right.\quad\quad\gamma_{t}=\left\{\begin{aligned}&t^{b_{1}},~~~&t\leq t_{1},\\&T^{b_{2}},~~&t>t_{1}.\end{aligned}\right.\end{equation*}where $\alpha_{t_{1}}\gamma_{t_{1}}\leq\alpha_{t_{1}+1}\gamma_{t_{1}+1}$, $a_{1}<a_{2}<-1$, and $b_{1}<b_{2}<\min\left\{-2,-3-2\eta,-2-2s\right\}$},
\item{$\left\{\beta_{t}\right\}_{t \in \left[T\right]}$ is a constant step size satisfying $\beta_{t}=T^{\beta}$with $\beta<-1-\eta$} and $0<\beta_{t}\eta_{t}<\frac{1}{2}$,
\item{$\left\{\zeta_{t}\right\}_{t \in \left[T\right]}$ is a positive step size satisfying $\zeta_{t}=1-(t+1)^{\zeta}$ with $\zeta<-1$}.
\end{enumerate}
\end{assumption}

\begin{theorem}\label{Theorem_3}
\rm{Under Assumptions \ref{assp.2}-\ref{assp.13}, if $\min \limits_{q\in\mathcal H_{i}}|\mathcal N^{i}_{iq}|>\frac{|\mathcal H_{i}|}{2}+2b_i-1$, then $\mathcal R_{\mathcal H}(T)$ and $\mathcal{CV}_{\mathcal H}(T)$ grow sublinearly in expectation.}
\end{theorem}
{\bf Proof.}
The results are proved by repeatedly using the fact that $\sum_{t=1}^{T}t^{c}\leq T^{1+c}$, $\forall -1<c<0$.\hfill$\blacksquare$

\begin{remark}
\rm{Theorem \ref{Theorem_3} demonstrates that the decentralized algorithm proposed in this paper can achieve sublinear growth of both $\mathcal R_{\mathcal H}(T)$ and $\mathcal{ CV}_{\mathcal H}(T)$ in expectation with appropriate step sizes. These results are of high interest for the online games. Compared to \cite{Sahoo_2021} and \cite{Dong_2024}, our approach not only accounts for randomness and multi-cluster games but also achieves sublinear regret growth in decentralized settings.} 
\end{remark}

\section{Numerical Experiments}\label{sec.four}

In this section, the robustness of the DBROSA algorithm against various Byzantine attacks is evaluated based on a commodity market game.

\subsection{Commodity Market Games}\label{sec.four.01}
In the commodity market game, $N$ parent companies produce commodities sold over $m$ markets. Each parent company $i \in [N]$ is composed of $n_{i}$ subsidiaries. Each subsidiary $j \in [n_{i}]$ in parent company $i$ participates in at most $m$ markets and decides on the quantity $x_{ij,t}\in \mathbb R^{d}$ of commodities to be delivered to the markets. Because the productivity of each company is limited, the decision variable $x_{ij,t}$ is bounded by the local constraint $\Omega_{ij}=\left\{x_{ij,t}\in \mathbb R^{d}|\bm 0_{d}\leq x_{ij,t}\leq X_{ij}\right\}$. Moreover, the capacity of each market $k$ also has an upper bound $b_{k}$ so that the coupled constraint function is given by $Ax_{t} \leq \omega_{t}b$ where $\omega_{t}\in\mathbb R$, $x_{t}=col(x_{ij,t})_{i\in[N],j\in[n_{i}]}\in \mathbb R^{nd}$, $b=col(b_{k})_{k\in[m]}\in \mathbb R^{m}$, and $A=[A_{11},\cdots, A_{Nn_{N}}]\in \mathbb R^{m\times nd}$ with $A_{ij}\in\mathbb R^{m\times d}$ denoting the situation that subsidiary $j$ in parent company $i$ participates in the markets. Specifically, $[A_{ij}]_{pq}=1$ if the $q$-th commodity produced by subsidiary $j$ in company $i$ is sent to the $p$-th market and $[A_{ij}]_{pq}=0$, otherwise.

Each subsidiary $j$ in parent company $i$ aims to minimize the cost of producing commodities and maximize the profit from selling them. Thus, the local cost function of each subsidiary $j$ in parent company $i$ is given by
\begin{align*}
f^{i,t}_{j}(x_{i,t},x_{-i,t},\theta_{i,t})=c^{i,t}_{j}(x_{ij,t})-P(Ax_t,\theta_{i,t})^{\top}A_{ij}x_{ij,t},
\end{align*}
where $c^{i,t}_{j}(x_{ij,t})=x_{ij,t}^{\top}Q_{ij,t}x_{ij,t}+q_{ij,t}^{\top}x_{ij,t}$ is a strongly convex cost function of production with $Q_{ij,t}$ and $q_{ij,t}$ denoting a time-varying diagonal matrix and a time-varying $d$-dimensional vector. Furthermore, $P(Ax_t,\theta_{i,t})$ is a linear pricing function, which depends on the total of commodities sold to each market and random factors in the market, meaning that
\begin{align*}
P(Ax_t,\theta_{i,t})=\bar P-D(\theta_{i,t})Ax_t,
\end{align*}
where $\bar P=col(\bar P_{k})_{k \in [m]}$ and $D(\theta_{i,t})=diag(d_{k}(\theta_{i,t}))_{k\in[m]}$.

The cost function of parent company $i$ is defined as $F_{i,t}(x_{i,t},x_{-i,t})=\frac{1}{|\mathcal H_i|}\sum_{j\in\mathcal H_{i}}\mathbf E_{\theta_{i,t}}[f^{i,t}_{j}(x_{i,t},x_{-i,t},\theta_{i,t})]$ with  $|\mathcal H_{i}|$ representing the number of honest subsidiaries in each parent company $i$. It can be seen that the above game model is equivalent to the stochastic online multi-cluster game (\ref{eq.BROMG}), where each parent company is regarded as a cluster and each subsidiary is viewed as an agent within the cluster.

\subsection{Different Byzantine Attacks}\label{bb}
The performance of the proposed algorithm is tested under four typical Byzantine attacks: Gaussian, max-value, sign-flipping, and sample-duplicating attacks.

To ease exposition, denote the message transmitted by subsidiary $j$ in each parent company $i$ as $m_{ij,t}$. Let~$W_i$ be the adjacency matrix corresponding to graph~$\mathcal G_i$. For Gaussian attacks, a Byzantine subsidiary $b \in \mathcal B_{i}$ generates its message $m_{ib,t}$ from a Gaussian distribution with mean $\frac{\sum_{h\in \mathcal H_{i}}[W_{i}]_{bh}m_{ih,t}}{\sum_{h\in \mathcal H_{i}}[W_{i}]_{bh}}$ and variance $u_{1}> 0$. For max-value attacks, a Byzantine subsidiary $b \in \mathcal B_{i}$ sets its message as $m_{ib,t}=u_{2}$. For sign-flipping attacks, a Byzantine subsidiary $b \in \mathcal B_{i}$ sets its message as $m_{ib,t}=u_{3}\cdot \frac{\sum_{h\in \mathcal H_{i}}[W_{i}]_{bh}m_{ih,t}}{\sum_{h\in \mathcal H_{i}}[W_{i}]_{bh}}$ with $u_{3}< 0$. For sample-duplicating attacks, a Byzantine subsidiary $b \in \mathcal B_{i}$ chooses $m_{ib,t}$ randomly from $\left\{u_{4}\cdot m_{ih,t}\right\}_{h\in \mathcal H_{i}}$.

\subsection{Small-Scale Market Games}
The parameters are set as: $N=3$, $n_{1}=n_{2}=n_{3}=5$, $|\mathcal B|=1$, $|\mathcal H|=12$, $|\mathcal H_{1}|=4$, $|\mathcal H_{2}|=5$, $|\mathcal H_{3}|=5$, $s(t)=1$, $m=d=4$, $X_{ij}=[20,20,20,20]^\top$, $A_{ij}=I_{4\times4}$, $Q_{ij,t}=(1+i+j+t^{-5})I_{4\times4}$, $\bar D=col(m+1)_{m\in[4]}$, $q_{ij,t}=col(0.1(m+i+j+t^{-5}))_{m\in[4]}$, $b_{i}=1$, $d_{k}(\theta_{i,t})$  is sampled from a uniform distribution centered at $0.8$ with bounded variance, $w_{t}$ is sampled from a uniform distribution centered at $2$ with bounded variance, then $10$ Monte-Carlo simulations are ran and the average results with $u_{1}=10$, $u_{2}=10000$, $u_{3}=-1$ and $u_{4}=-100$ over the fully connected graph are plotted. As shown in Figs. \ref{Fig_1} and \ref{Fig_2}, $\mathcal R_{\mathcal H}(T)$ and $\mathcal{CV}_{\mathcal H}(T)$ grow sublinearly under any Byzantine attacks in \ref{bb}. 

\begin{figure}[H]
\centering
\includegraphics[scale=0.28]{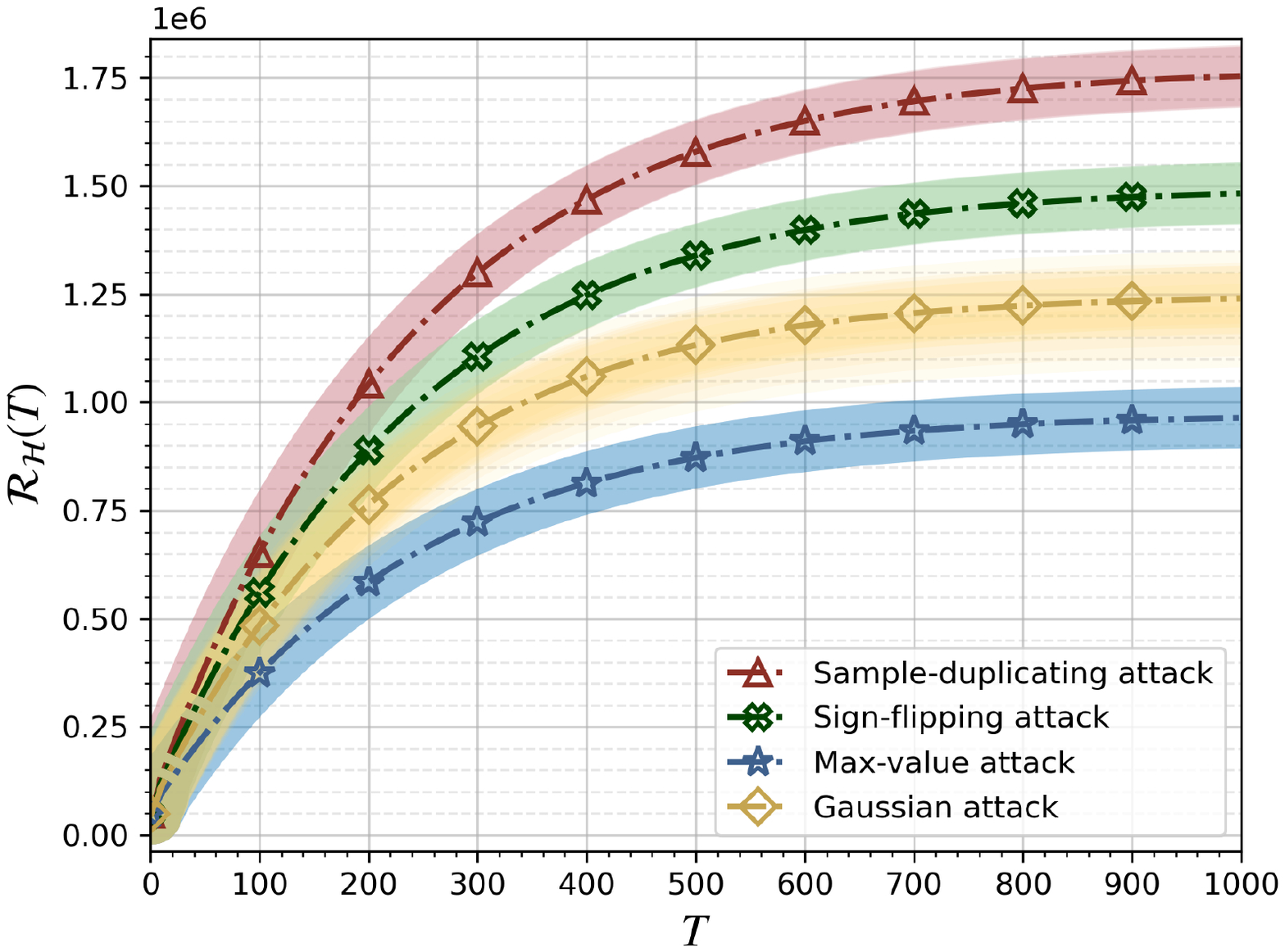}
\caption{The trajectories of $\mathcal R_{\mathcal H}(T)$ under different Byzantine attacks.}\label{Fig_1}
\end{figure}

\begin{figure}[H]
\centering
\includegraphics[scale=0.28]{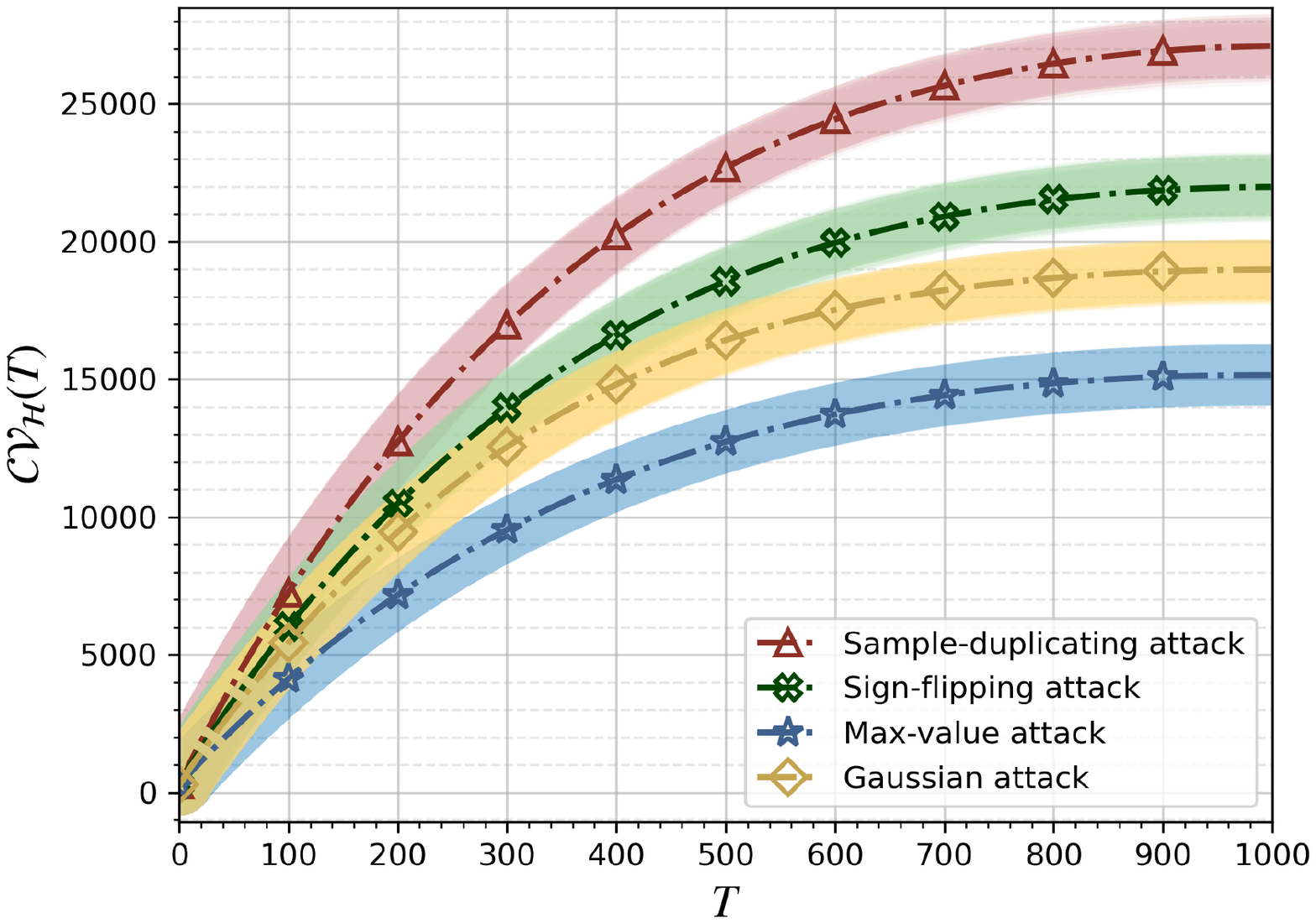}
\caption{The trajectories of $\mathcal{CV}_{\mathcal H}(T)$ under different Byzantine attacks.}\label{Fig_2}
\end{figure}

\begin{figure}[H]
\centering
\includegraphics[scale=0.1]{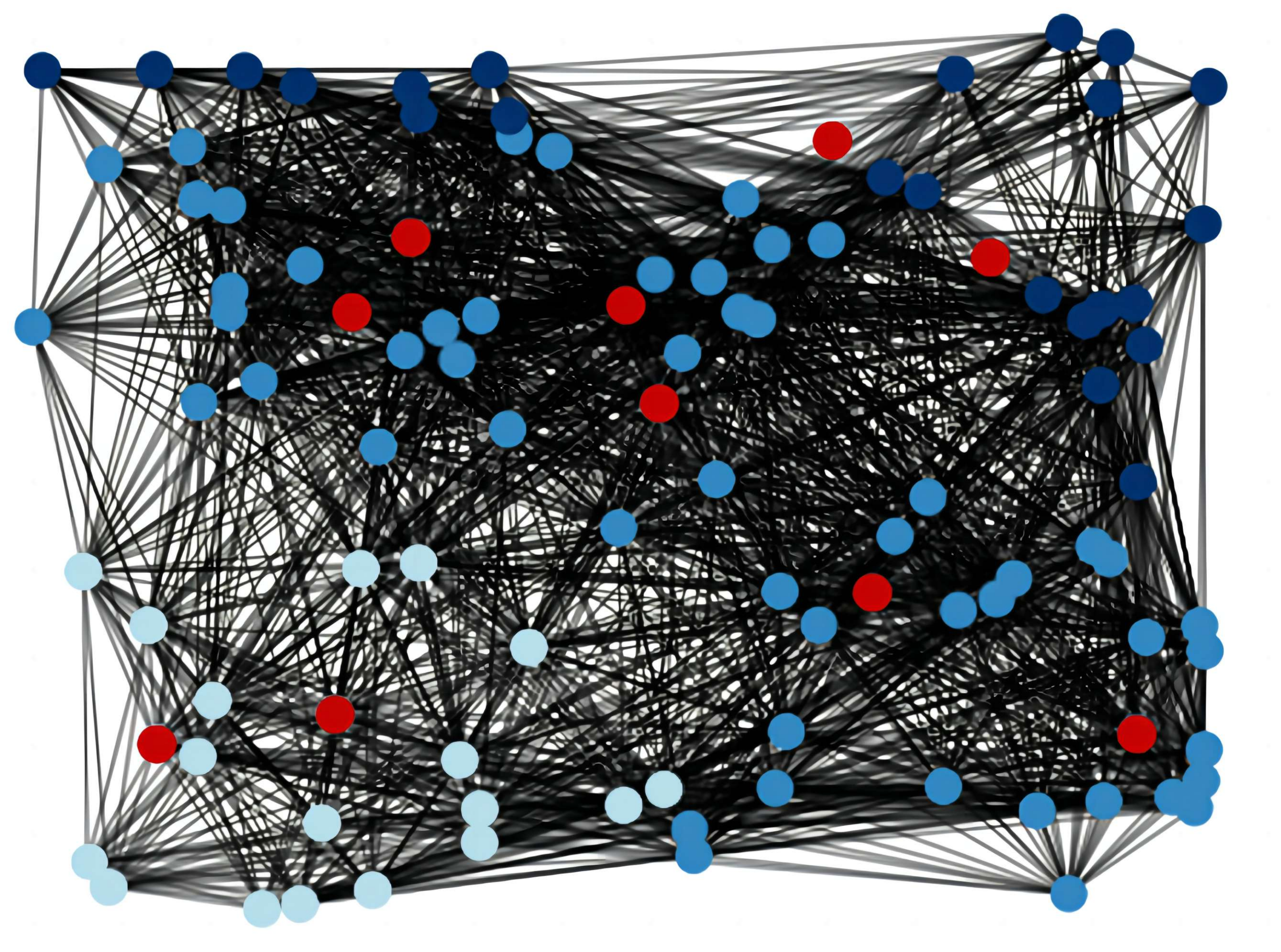}
\caption{The complex communication network between 100 companies.}\label{Fig_3}
\end{figure}

\subsection{Large-Scale Market Games}
Consider the commodity market game that consists of $100$ companies, as depicted in Fig. \ref{Fig_3}, where each circle signifies a company. Companies of the same color denote that they belong to the same cluster, and companies marked in red represent the Byzantine companies. The parameters are set as: $N=100$, $n_{1}=20$, $n_{2}=55$, $n_{3}=25$, $|\mathcal B|=10$, $|\mathcal H|=90$, $|\mathcal H_{1}|=18$, $|\mathcal H_{2}|=49$, $|\mathcal H_{3}|=23$, $s(t)=1$, $m=d=1$, $X_{ij}=5$, $A_{ij}=1$, $Q_{ij,t}=1+i+j+t^{-5}$, $q_{ij,t}=0.1(i+j+t^{-5})$, $b_{i}=3$, $d_{k}(\theta_{i,t})$ is sampled from a Gaussian distribution centered at $0.8$ with bounded variance, $w_{t}$ is sampled from a Gaussian distribution centered at $2$ with bounded variance. Then, the simulation results with $u_{1}=5$, $u_{2}=100$, $u_{3}=-0.5$ and $u_{4}=-100$ are plotted over Figs. \ref{Fig_4}-\ref{Fig_6}. As shown in Figs. \ref{Fig_4} and \ref{Fig_6}, $\mathcal R_{\mathcal H}(T)$  and $\mathcal{CV}_{\mathcal H}(T)$ grows sublinearly. 

\begin{figure}[H]
\centering
\includegraphics[scale=0.3]{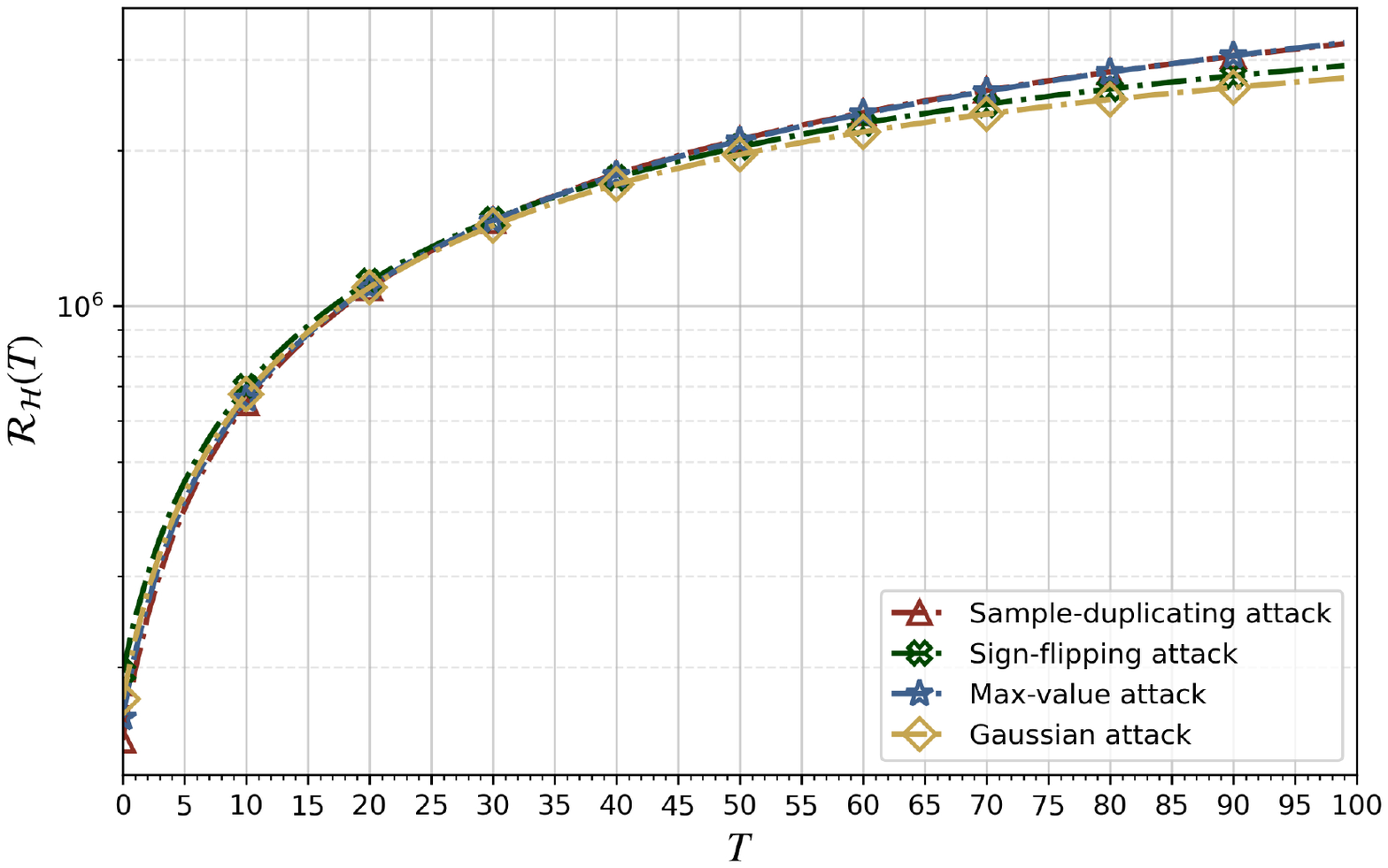}
\caption{The trajectories of $\mathcal R_{\mathcal H}(T)$ related to different Byzantine attacks over $100$ companies.}\label{Fig_4} 
\end{figure}

\begin{figure}[H]
\centering
\includegraphics[scale=0.195]{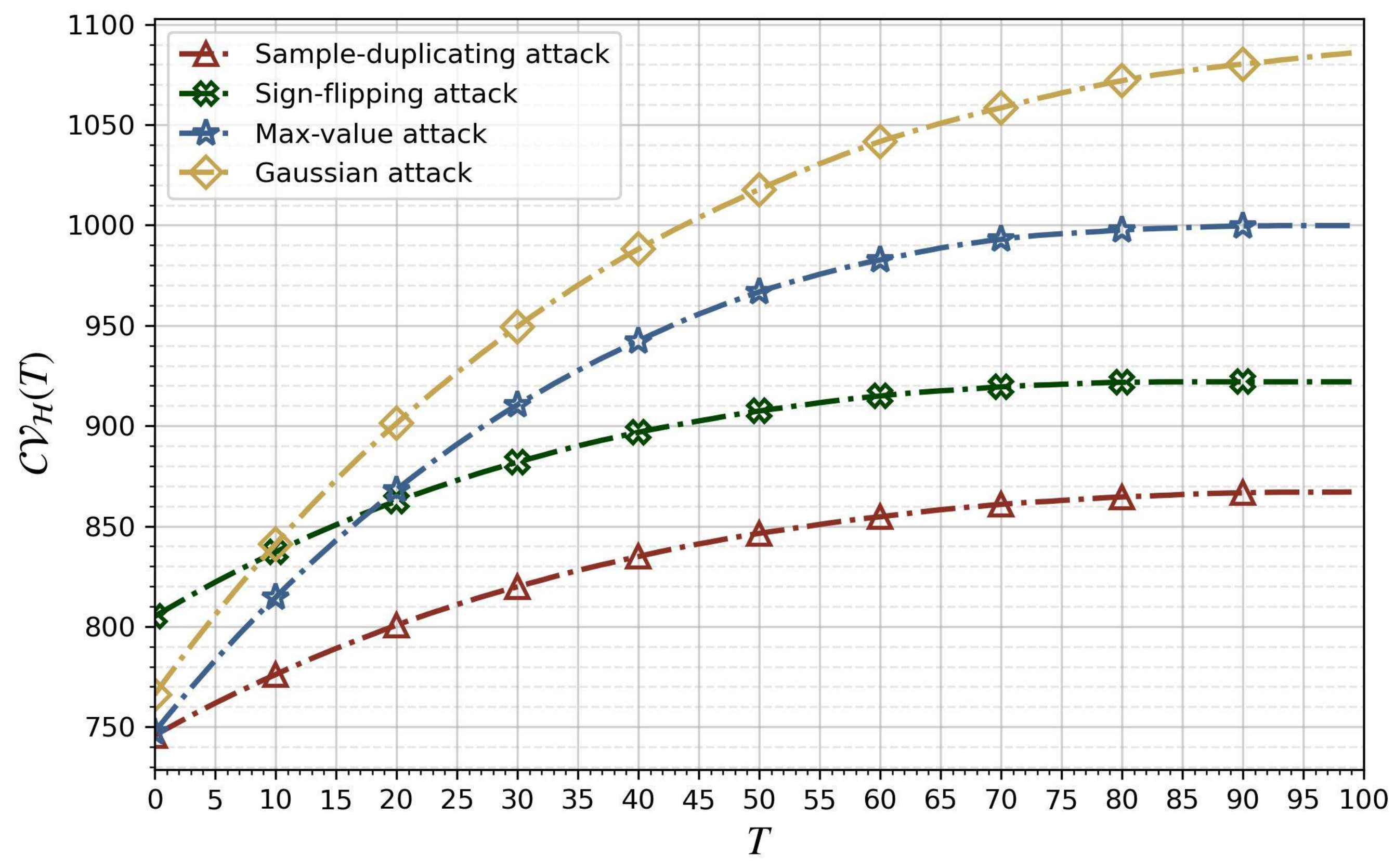}
\caption{The trajectories of $\mathcal{CV}_{\mathcal H}(T)$ related to different Byzantine attacks over $100$ companies.}\label{Fig_6} 
\end{figure}

\section{Conclusion}\label{sec.five}
This paper investigates the online multi-cluster game involving randomness under Byzantine attacks for the first time. Aiming to effectively distinguish the impact of randomness from Byzantine attacks, a novel decentralized Byzantine-resilient online SGNE seeking algorithm is proposed, which combines mod-SVRG technique, dynamic average consensus, and a robust aggregation mechanism. Furthermore, the resilient system-wise regret $\mathcal R_{\mathcal H}(T)$ and constraint violation $\mathcal{CV}_{\mathcal{H}}(T)$ are first defined within an online game framework under Byzantine attacks. These metrics are designed to measure the convergence performance of the algorithm and aim to prevent the decisions of Byzantine agents from being included in the metrics. The theoretical analysis verifies that  $\mathcal R_{\mathcal H}(T)$ and $\mathcal{CV}_{\mathcal{H}}(T)$ can grow sublinearly in expectation. Finally, the numerical experiments fully demonstrate that the algorithm is robust to Byzantine attacks.

\begin{appendix}
\linespread{0.2}
\subsection{The convergence analysis of mod-SVRG}\label{mod-SVRG}
The objective of mod-SVRG is to obtain the optimal solution to the following optimization problem.
\begin{equation}\label{svrg}
\min_{w} P(w) = \mathbf{E}_{\xi \sim \mathcal D} \left[ f(w, \xi) \right],
\end{equation}
where $\mathcal D$ denotes the probability distribution that the random variable $\xi$ follows. We assume that $\mathbf{E}_{\xi}\left[\nabla f(w, \xi)\right]=\nabla P(\omega)$. The decision variable $\omega_s$ generated by mod-SVRG can be updated by Algorithm 2.

\begin{algorithm}\label{alg:mod-svrg}
\caption{Modified SVRG (mod-SVRG) for General Stochastic Optimization}
\textbf{Input:} Learning rate $\eta$, update frequency $\{m_s\}_{s=1}^{\infty}$, and batch size sequence $\{B_s\}_{s=1}^{\infty}$.

\textbf{Initialize:} $\tilde{w}_0$.

\textbf{for} $s = 1, 2, \ldots$ \textbf{do}
\begin{enumerate}
\item[] Compute gradient estimator: 
\item[] \quad\quad\quad$\hat{\mu}_s = \frac{1}{B_s} \sum_{k=1}^{B_s} \nabla f(\tilde{w}_{s-1}, \xi_k)$ \quad ($\xi_k \overset{\text{i.i.d.}}{\sim} \mathcal{D})$.
\item[] Initialize inner iterate: $w_0 =\tilde{w}_{s-1}$.
\item[] \textbf{for} $t = 1, 2, \ldots, m_s$ \textbf{do}
\item[] \quad Randomly sample: $\xi_t \sim \mathcal{D}$.
\item[] \quad  Compute variance-reduced gradient:
\item[]  \quad\quad\quad\quad$g_t=\nabla f(w_{t-1}, \xi_t) - \nabla f(\tilde{w}_{s-1}, \xi_t) + \hat{\mu}_s$.
\item[]  \quad Update decision variable $\omega_t$:
\item[]   \quad\quad\quad\quad\quad\quad\quad\quad $w_t = w_{t-1} - \eta g_t$.
\item[] \textbf{endfor}
\item[]  \textbf{Option I:} Set $\tilde{w}_s = w_{m_s}$.
\item[]  \textbf{Option II:} Set $\tilde{w}_s =w_t$ for randomly chosen $t \in \{0, \ldots, m_s-1\}$.
\end{enumerate}
\textbf{end for}
\end{algorithm}

\begin{theorem}\label{thm:mod-svrg}
\rm{Consider mod-SVRG with option II for problem (\ref{svrg}). Let $\mathcal{F}_t$ be the $\sigma$-field generated by the random variables $\xi_t$. Suppose that $P(w)$ is $\gamma$-strongly convex and $\nabla f(w,\xi)$ is $L$-smooth. Furthermore, assume that the gradient estimator $\hat{\mu}_s$ is unbiased conditioned on $\mathcal F_t$ and the variance is bounded, i.e., $\mathbf{E}_{\mathcal{F}_t}[\hat{\mu}_s] = \nabla P(\tilde{w}_{s-1})$ and  $\mathbf{E}_{\mathcal{F}_t}\|\hat{\mu}_s - \nabla P(\tilde{w}_{s-1})\|^2 \leq \tilde{\sigma}^2$. In addition, let the learning rate satisfy $\eta < \frac{1}{4L}$, and the update frequency $m_s$ be chosen such that
\[
\alpha_s \triangleq \frac{1}{\gamma\eta(1 - 4L\eta)m_s} + \frac{4L\eta}{1 - 4L\eta} < 1.
\]
Then mod-SVRG achieves geometric convergence in expectation, i.e.,
\[
\mathbf{E}[P(\tilde{w}_s) - P(w_*)] \leq \alpha^{s} \mathbf{E}[P(\tilde{w}_{0}) - P(w_*)] + \frac{(1-\alpha^{s})\beta}{1-\alpha},
\]
where $\alpha=\mathop{\max}\limits_{s}\left\{\alpha_s\right\}$, $\beta = \frac{\eta\tilde{\sigma}^2}{1 - 4L\eta}$, and $w_* = \arg\min_{w} P(w)$.}
\end{theorem}

\begin{proof}
Decompose $g_t$ as follows.
\begin{align*}
g_t = &\underbrace{\nabla f(w_{t-1},\xi_t) - \nabla f(\tilde{w}_{s-1},\xi_t) + \nabla P(\tilde{w}_{s-1})}_{g_t^{\text{SVRG}}}
\\&+ \underbrace{(\hat{\mu}_s - \nabla P(\tilde{w}_{s-1}))}_{\epsilon_s}.
\end{align*}

Taking the squared norm on both sides of the above equality and then taking expectation with respect to $\mathcal F_t$, one has
\begin{align*}
\mathbf{E}_{\mathcal F_t}\|g_t\|^2 
&\leq 2\mathbf{E}_{\mathcal F_t}\|g_t^{\text{SVRG}}\|^2 + 2\mathbf{E}_{\mathcal F_t}\|\epsilon_s\|^2 \\
&\leq 8L[P(w_{t-1}) - P(w_*) + P(\tilde{w}_{s-1}) - P(w_*)] + 2\tilde{\sigma}^2,
\end{align*}
where the last inequality results from the analysis of Theorem 1 in \cite{Johnson_2013}.

By the assumption that $\mathbf{E}_{\mathcal F_t}[\hat{\mu}_s] = \nabla P(\tilde{w}_{s-1})$, one has $\mathbf{E}_{\mathcal F_t}[g_t]= \nabla P(w_{t-1})$, which lead to
\begin{align*}
&\mathbf{E}_{\mathcal F_t}\|w_t - w_*\|^2 \\
&= \|w_{t-1} - w_*\|^2 - 2\eta(w_{t-1} - w_*)^\top \mathbf{E}_{\mathcal F_t}[g_t] + \eta^2 \mathbf{E}_{\mathcal F_t}\|g_t\|^2 \\
&\leq \|w_{t-1} - w_*\|^2 - 2\eta(w_{t-1} - w_*)^\top \nabla P(w_{t-1}) \\
&\quad + \eta^2 \left(8L[P(w_{t-1}) - P(w_*) + P(\tilde{w}_{s-1}) - P(w_*)] + 2\tilde{\sigma}^2\right) \\
&\leq \|w_{t-1} - w_*\|^2 - 2\eta(1 - 4L\eta)[P(w_{t-1}) - P(w_*)] \\
&\quad + 8L\eta^2[P(\tilde{w}_{s-1}) - P(w_*)] + 2\eta^2\tilde{\sigma}^2,
\end{align*}
where the first inequality follows from the upper bound of $\mathbf{E}_{\mathcal F_t}\|g_t\|^2 $ and the last inequality results from the convexity of $P(\omega)$.

Summing the above inequality over $t=1,\cdots,m_s$, applying strong convexity $\|w - w_*\|^2 \leq \frac{2}{\gamma}[P(w) - P(w_*)]$, and taking expectation with respect to $\mathscr F_{T}$ with $\mathscr F_{t}=\cup_{s=1}^{t}\mathcal F_s$, one has
\begin{align*}
&\sum_{t=0}^{m_s-1} \mathbf{E}[P(w_t) - P(w_*)] 
\\&\leq \frac{1+4\gamma L\eta^{2}m_s}{\gamma\eta(1-4L\eta)}\mathbf{E}[P(\tilde{w}_{s-1}) - P(w_*)] + \frac{\eta m_s\tilde{\sigma}^2}{1-4L\eta}.
\end{align*}

If $\tilde{w}_s$ is chosen by OptionII, the following relation holds.
\begin{align*}
\mathbf{E}[P(\tilde{w}_s) - P(w_*)] \leq \alpha_s \mathbf{E}[P(\tilde{w}_{s-1}) - P(w_*)] + \beta,
\end{align*}
where $\alpha_s = \frac{1}{\gamma\eta(1-4L\eta)m_s}+\frac{4L\eta}{1-4L\eta}$ and $\beta = \frac{\eta\tilde{\sigma}^2}{1-4L\eta}$. Let $\alpha=\frac{1}{\gamma\eta(1-4L\eta)\mathop{\min}\limits_{s}\left\{m_s\right\}}+\frac{4L\eta}{1-4L\eta}$, then the proof is complete.\hfill$\blacksquare$
\end{proof}

The key difference between SVRG and mod-SVRG lies in the fact that SVRG converges exactly to the optimal solution under suitable conditions, whereas mod-SVRG only guarantees convergence to a neighborhood of the optimal solution. The convergence behavior of mod-SVRG is governed by the gradient estimation error $\tilde{\sigma}$, the update frequency $m_s$, the step size $\eta$, and inherent mathematical properties of the functions. Specifically, an increase in $\tilde{\sigma}$ increases $\beta$, expanding the neighborhood range, while an increase in $m_s$ decreases $\alpha$, thereby improving the convergence rate and reducing the neighborhood range. Notably, the gradient estimator $\hat{\mu}_s$ in mod-SVRG must satisfy two key assumptions: unbiasedness and bounded variance, while the Algorithm 2 uses a  mini-batch gradient estimator \cite{Bottou_2018}, any other estimator meeting these conditions can be substituted without affecting the theoretical convergence guarantees. Clearly, when $\tilde \sigma=0$, the variance of $g_t$ in estimating the global gradient $\nabla P(\omega_{t-1})$ tends to $0$ as $s\rightarrow \infty$. This convergence occurs because $\mathbf E_{\mathcal F_t}\left\|g_t-\nabla P(\omega_{t-1})\right\|^{2}\rightarrow0$ , driven by the  convergence of both $\omega_{t-1}$ and $\tilde\omega_{s-1}$ to $\omega_*$ as $s\rightarrow \infty$. Nevertheless, when $\tilde \sigma\neq0$, if $m_s=1$ as $s\rightarrow \infty$, then $\mathbf E_{\mathcal F_t}\left\|g_t-\nabla P(\omega_{t-1})\right\|^{2}\rightarrow0$ still holds due to $\omega_{t-1}=\tilde\omega_{s-1}$ as $s\rightarrow \infty$.
\subsection{Proof of Lemma \ref{lemma_2}}\label{prof_lemma_2}
For any $q \in \mathcal H_{p}$, let $\bm e^{pqk}_{t}=col(e^{pqk}_{ij,t})_{i\in [N],j\in \mathcal H_{i},(i,j)\neq(p,q)}$ with $e^{pqk}_{ij,t}$ denoting $[\bm x^{pq}_{ij,t}]_{k}-[x_{pq,t}]_{k}$. Let $|\mathcal H|^{-}\triangleq|\mathcal H|-1$. Then, by (\ref{row_form_4}), one has
\begin{align*}
\bm e^{pqk}_{t+1}&=\delta_{t}\bm M_{pq,t}^{1k-}\bm e^{pqk}_{t}+\bm 1_{|\mathcal H|^{-}}([x_{pq,t}]_{k}-[x_{pq,t+1}]_{k})
\\&\quad+(\zeta_{t}-\delta_{t})\bm e^{pqk}_{t}-(1-\zeta_{t})\bm 1_{|\mathcal H|^{-}}[x_{pq,t}]_{k},
\\&=\left((\zeta_{t}-\delta_{t})I_{|\mathcal H|^{-}}+\delta_{t}\bm M_{pq,t}^{1k-}\right)\bm e^{pqk}_{t}
\\&\quad+\bm 1_{|\mathcal H|^{-}}([x_{pq,t}]_{k}-[x_{pq,t+1}]_{k})-(1-\zeta_{t})\bm 1_{|\mathcal H|^{-}}[x_{pq,t}]_{k}.
\end{align*}

Let $v_{t}=\left\|(\zeta_{t}-\delta_{t})I_{|\mathcal H|^{-}}+\delta_{t}\bm M_{pq,t}^{1k-}\right\|$, then $v_{t}<1$ for any $t \in [T]$ due to the properties of $\zeta_{t}$, $\delta_{t}$ and the fact that $\left\|\bm M_{pq,t}^{1k-}\right\|\leq\sqrt{\left\|\bm M_{pq,t}^{1k-}\right\|_{1}\left\|\bm M_{pq,t}^{1k-}\right\|_{\infty}}$ holds. Define $v_{1}=\max \limits_{t\in[T]}v_{t}$, then one obtains
\begin{equation}\label{eq.61}
\begin{aligned}
\left\|\bm e^{pqk}_{t+1}\right\|&\leq v_{1}\left\|\bm e^{pqk}_{t}\right\|+\sqrt{|\mathcal H|}\left\|[x_{pq,t+1}]_{k}-[x_{pq,t}]_{k}\right\|
\\&\quad+(1-\zeta_{t})\sqrt{|\mathcal H|}\left\|[x_{pq,t}]_{k}\right\|.
\end{aligned}
\end{equation}

From the update rule (\ref{eq.3}), one has
\begin{equation}\label{eq.62}
\left\|[x_{pq,t+1}]_{k}-[x_{pq,t}]_{k}\right\|\leq \left\|x_{pq,t+1}-x_{pq,t}\right\|\leq 2\alpha_{t}R.
\end{equation}

Plugging (\ref{eq.62}) into (\ref{eq.61}), one obtains
\begin{equation}\label{eq.63}
\begin{aligned}
\left\|\bm e^{pqk}_{t+1}\right\|&\leq v_{1}\left\|\bm e^{pqk}_{t}\right\|+2\alpha_{t}\sqrt{|\mathcal H|}R+(1-\zeta_{t})\sqrt{|\mathcal H|}R.
\end{aligned}
\end{equation}

Applying telescopic cancellation to (\ref{eq.63}) yields,
\begin{equation}\label{eq.64}
\begin{aligned}
\left\|\bm e^{pqk}_{t+1}\right\|&\leq v_{1}^{t}\left\|\bm e^{pqk}_{1}\right\|+2\sqrt{|\mathcal H|}R\sum_{r=0}^{t-1}v^{r}_{1}\left(\alpha_{t-r}+(1-\zeta_{t-r})\right).
\end{aligned}
\end{equation}

Thus, let $v_{2}=\max \limits_{t\in[T]}\left\|(\zeta_{t}-\delta_{t})I_{|\mathcal H|}+\delta_{t}\bm M_{pq,t}^{2k}\right\|$ and $v=\max \left\{v_{1},v_{2}\right\}$, combining the following relation with (\ref{eq.64}), then (\ref{eq.01}) can be obtained.
\begin{align*}
\left\|\bm x^{pq}_{ij,t}-x_{pq,t}\right\|\leq \sum_{k=1}^{d}\left\|[\bm x^{pq}_{ij,t}]_{k}-[x_{pq,t}]_{k}\right\|\leq\sum_{k=1}^{d}\left\|\bm e^{pqk}_{t}\right\|.
\end{align*}

The proof of (\ref{eq.02}) generally follows the similar lines of (\ref{eq.01}) by transforming $[x_{pq,t}]_{k}$ into $[x_{ph,t}^{*}]_{k}$, which is omitted here due to the space limit.\hfill$\blacksquare$

\subsection{Proof of Lemma \ref{lemma_3}}\label{prof_lemma_3}
Using telescopic cancellation on (\ref{row_form_1}) and denoting $\bm \epsilon_{iq,t}^{k}-\bm \epsilon_{iq,t-1}^{k}$ as $\Delta\bm \epsilon_{iq,t}^{k}$. Then, 
\begin{equation}\label{eq.20}
\begin{aligned}
\bm v_{iq,t}^{k}=\Phi^{k}_{iq,[t-1,1]}\bm v_{iq,1}^{k}+\sum_{g=2}^{t}\Phi^{k}_{iq,[t-1,g]}\Delta\bm \epsilon _{iq,g}^{k},
\end{aligned}
\end{equation}
where $\Phi^{k}_{iq,[t-1,g]}\triangleq\bm Y_{iq,t-1}^{k}\times\cdots\times\bm Y_{iq,g}^{k}$ and $\Phi^{k}_{iq,[t-1,t]}\triangleq I_{d}$. All of them have the following properties, i.e., there exist some constants $C>0$, $\Theta\in (0,1)$ and stochastic vector $p^{k}_{iq,g}$ with all elements upper bounded by $\frac{1}{\min \limits_{q\in\mathcal H_{i}}\left\{|\mathcal N^{i}_{iq}|-2b_i+1\right\}}$ such that 
\begin{equation}\label{eq.21}
\begin{aligned}
\left\|\Phi^{k}_{iq,[t-1,g]}-\bm 1_{|\mathcal H_{i}|}p^{k\top}_{iq,g}\right\|\leq C\Theta^{t-g}.
\end{aligned}
\end{equation}
Please refer to \cite{Nedic_2010} for a more detailed discussion.

By the definition of $\bm v_{iq,t}$, it can be obtained that
\begin{equation}\label{eq.12}
\begin{aligned}
&\sum_{t=2}^{T}\left\|\bm v_{iq,t}-(\bm 1_{|\mathcal H_{i}|}\otimes I_{d})\bar \epsilon_{iq,t}\right\|
\\&\leq \sum_{t=2}^{T}\sum_{k=1}^{d}\left\|\bm v_{iq,t}^{k}-\bm 1_{|\mathcal H_{i}|}[\bar \epsilon_{iq,t}]_{k}\right\|
\\&\leq \sum_{t=2}^{T}\sum_{k=1}^{d}\left\|\bm v_{iq,t}^{k}-\hat{\bm v}_{iq,t}^{k}+\hat{\bm v}_{iq,t}^{k}-\bm 1_{|\mathcal H_{i}|}[\bar \epsilon_{iq,t}]_{k}\right\|
\\&\leq \sum_{t=2}^{T}\sum_{k=1}^{d}\left\|\bm v_{iq,t}^{k}-\hat{\bm v}_{iq,t}^{k}\right\|+\sum_{t=2}^{T}\sum_{k=1}^{d}\left\|\hat{\bm v}_{iq,t}^{k}-\bm 1_{|\mathcal H_{i}|}[\bar \epsilon_{iq,t}]_{k}\right\|,
\end{aligned}
\end{equation}
where $\hat{\bm v}_{iq,t}^{k}=\bm 1_{|\mathcal H_{i}|}\sum_{g=2}^{t}p^{k \top}_{iq,g}\Delta \bm \epsilon^{k}_{iq,g}$.

Thus, 
\begin{equation}\label{eq.13}
\begin{aligned}
&\left\|\bm v_{iq,t}^{k}-\hat{\bm v}_{iq,t}^{k}\right\|
\\&\leq \sum_{g=2}^{t}\left\|\Phi^{k}_{iq,[t-1,g]}-\bm 1_{|\mathcal H_{i}|}p^{k \top}_{iq,g}\right\|\left\|\Delta \bm \epsilon^{k}_{iq,g}\right\|
\\&\leq C\sum_{g=2}^{t}\Theta^{t-g}\left\|\Delta \bm \epsilon^{k}_{iq,g}\right\|,
\end{aligned}
\end{equation}
where the first inequality follows from the fact that $v^{j}_{iq,1}=\epsilon^{j}_{iq,1}=\bm 0_{d}$ for any $j \in \mathcal H_{i}$.

Furthermore, let $v_{\mathcal G}\triangleq\Big( \frac{|\mathcal H_{i}|}{\min \limits_{q\in\mathcal H_{i}}\left\{|\mathcal N^{i}_{iq}|-2b_i+1\right\}}-1\Big)^{\frac{1}{T}}<1$ and $\min \limits_{q\in\mathcal H_{i}}|\mathcal N^{i}_{iq}|>\frac{|\mathcal H_{i}|}{2}+2b_i-1$, it holds that
\begin{equation}\label{eq.16}
\begin{aligned}
&\left\|\hat{\bm v}_{iq,t}^{k}-\bm 1_{|\mathcal H_{i}|}[\bar \epsilon_{iq,t}]_{k}\right\|
\\&\leq \sum_{g=2}^{t}\left\|\bm 1_{|\mathcal H_{i}|}p^{k\top}_{iq,g}-\frac{1}{|\mathcal H_{i}|}\bm 1_{|\mathcal H_{i}|}\bm 1^{\top}_{|\mathcal H_{i}|}\right\|\left\|\Delta \bm \epsilon^{k}_{iq,g}\right\|
\\&\leq\sum_{g=2}^{t} v_{\mathcal G}^{t-g}\left\|\Delta \bm \epsilon^{k}_{iq,g}\right\|,
\end{aligned}
\end{equation}
where the second inequality follows from that $[\bar \epsilon_{ij,1}]_{k}=0$ holds for all $i \in[N]$, $j\in\mathcal H_{i}$, and $k \in [d]$.

Substituting (\ref{eq.13}) and (\ref{eq.16}) into (\ref{eq.12}) and taking the conditional expectation on both sides of (\ref{eq.12}), one has
\begin{equation}\label{eq.17}
\begin{aligned}
&\sum_{t=2}^{T}\mathbf E_{\mathcal U_t}\left\|\bm v_{iq,t}-(\bm 1_{|\mathcal H_{i}|}\otimes I_{d})\bar \epsilon_{iq,t}\right\|
\\&\leq C\sum_{t=2}^{T}\sum_{k=1}^{d}\sum_{g=2}^{t}\Theta^{t-g}\mathbf E_{\mathcal U_t}\left\|\Delta \bm \epsilon^{k}_{iq,g}\right\|
\\&\quad+\sum_{t=2}^{T}\sum_{k=1}^{d}\sum_{g=2}^{t} v_{\mathcal G}^{t-g}\mathbf E_{\mathcal U_t}\left\|\Delta \bm \epsilon^{k}_{iq,g}\right\|
\\&\leq \left(\frac{\sqrt{d}C}{1-\Theta}+\frac{\sqrt{d}}{1-v_{\mathcal G}}\right)\sum_{t=2}^{T}\mathbf E_{\mathcal U_t}\left\|\Delta \bm \epsilon_{iq,t}\right\|,
\end{aligned}
\end{equation}
where $\Delta \bm \epsilon_{iq,t}=\bm \epsilon_{iq,t}-\bm \epsilon_{iq,t-1}$ and $\bm \epsilon_{iq,t}=col(\epsilon^{j}_{iq,t})_{j\in\mathcal H_{i}}$ .

Next, an upper bound of $\sum_{t=2}^{T}\mathbf E_{\mathcal U_t}\left\|\Delta \bm \epsilon_{iq,t}\right\|$ can be constructed by  Lemma \ref{lemma_2} and Assumption \ref{assp.7} as follows.
\begin{equation}\label{eq.1118}
\begin{aligned}
&\sum_{t=2}^{T}\mathbf E_{\mathcal U_t}\left\|\Delta \bm \epsilon_{iq,t}\right\|
\\&\leq \sqrt{2}l_{f}\sum_{t=2}^{T}\sum_{j\in \mathcal H_{i}}\big(\left\|\bm x_{ij,t}-\bm \tau_{ij,t}\right\|+\left\|\bm x_{ij,t-1}-\bm \tau_{ij,t-1}\right\|
\\&\quad+\left\|\bm \tau_{ij,t}-\bm \tau_{ij,t-1}\right\|\big)+\sqrt{2}|\mathcal H|\Delta F^{\text{sup}}_{T}
\\&\leq 3\sqrt{2}t_{0}l_{f}\sum_{t=2}^{T}\sum_{j\in \mathcal H_{i}}\left\|\bm x_{ij,t}-\bm x_{ij,t-1}\right\|+\sqrt{2}|\mathcal H|\Delta F^{\text{sup}}_{T}
\\&\leq  3\sqrt{2}t_{0}l_{f}\sum_{t=2}^{T}\sum_{j\in \mathcal H_{i}}\Big(\left\|\bm x_{ij,t}-x_{\mathcal H,t}^{*}\right\|+\left\|\bm x_{ij,t-1}-x_{\mathcal H,t-1}^{*}\right\|
\\&\quad+\left\|x_{\mathcal H,t}^{*}-x_{\mathcal H,t-1}^{*}\right\|\Big)+\sqrt{2}|\mathcal H|\Delta F^{\text{sup}}_{T}
\\&\leq3\sqrt{2}t_{0}l_{f}\sum_{t=1}^{T}\sum_{j\in \mathcal H_{i}}\big(2\left\|\bm x_{ij,t}-x_{\mathcal H,t}^{*}\right\|+\left\|x_{\mathcal H,t}^{*}-x_{\mathcal H,t-1}^{*}\right\|\big)
\\&\quad+\sqrt{2}|\mathcal H|\Delta F^{\text{sup}}_{T}
\\&\leq6\sqrt{2}t_{0}l_{f}\sum_{t=1}^{T}\sum_{j\in \mathcal H_{i}}\sum_{p=1}^{N}\sum_{q\in\mathcal H_{p}}\sum_{h\in\mathcal B_{p}}\big(\left\|\bm x_{ij,t}^{pq}-x_{pq,t}\right\|
\\&\quad+\left\|\bm x_{ij,t}^{ph}-x_{ph,t}^{*}\right\|+\left\|x_{pq,t}-x_{pq,t-1}\right\|+\left\|x_{ph,t}^{*}-x_{ph,t-1}^{*}\right\|\big)
\\&\quad+\sqrt{2}|\mathcal H|\Delta F^{\text{sup}}_{T}
\\&\leq \mathcal O(1+\sum_{t=1}^{T}\alpha_{t}+\sum_{t=1}^{T}\xi_{t}+\Phi(T)+\Delta F^{\text{sup}}_{T}),
\end{aligned}
\end{equation}
where $x_{\mathcal H,0}^{*}\triangleq x_{\mathcal H,1}^{*}$ and $t_{0}$ is defined as $t_{0}=\max_{t\in \left\{t\in[T]|\bmod(t,s(t))\neq 0\right\}}\left\{\sum_{k=t}^{\min_{t^{*}}\left\{t^{*}>t|\bmod(t^{*},s(t^{*}))=0\right\}}1\right\}$.

Substituting (\ref{eq.1118}) into (\ref{eq.17}), the proof is complete.

\subsection{Proof of Theorem \ref{Theorem_1}}\label{prof_lemma_7}

Let $\Lambda=\max\left\{R,\left\{\left\|\bm x_{ij,1}\right\|\right\}_{i\in[N],j\in [\mathcal H_i]}\right\}$. For any $i,p\in[N]$, $j\in \mathcal H_{i}$, $q\in \mathcal H_{p}$ and $k\in[d]$, if $(i,j)= (p,q)$, it is natural to get $\left\|[\bm x_{ij,t}^{pq}]_{k}\right\|=\left\|[x_{ij,t}]_{k}\right\|\leq \Lambda$ by (\ref{eq.3}). On the other hand, if $(i,j)\neq (p,q)$, then
\begin{align*}
\left\|[\bm x_{ij,2}^{pq}]_{k}\right\|&\leq\delta_{1}\frac{\sum_{(r,h) \in \mathcal Y_{ij,1}^{pqk}}\left\|[\widetilde{\bm x}_{rh,1}^{pq}]_{k}\right\|}{|\mathcal N_{ij}|-2b+1}+(\zeta_{1}-\delta_{1})\left\|[\bm x_{ij,1}^{pq}]_{k}\right\|
\\&\leq \delta_{1}\left\|[\bm M^{1k}_{pq,t}\bm x^{pqk+}_{t}]_{\mathcal S(i,j)}\right\|+(\zeta_{1}-\delta_{1})\left\|[\bm x_{ij,1}^{pq}]_{k}\right\|
\\&\leq \zeta_{1}\Lambda<\Lambda.
\end{align*}

For any $i,p\in[N]$, $j\in \mathcal H_{i}$, $q\in \mathcal B_{p}$ and $k\in[d]$, one has
\begin{align*}
\left\|[\bm x_{ij,2}^{pq}]_{k}\right\|&\leq\delta_{1}\frac{\sum_{(r,h) \in \mathcal Y_{ij,1}^{pqk}}\left\|[\widetilde{\bm x}_{rh,1}^{pq}]_{k}\right\|}{|\mathcal N_{ij}|-2b+1}+(\zeta_{1}-\delta_{1})\left\|[\bm x_{ij,1}^{pq}]_{k}\right\|
\\&\leq \delta_{1}\left\|[\bm M^{2k}_{pq,t}\bm x^{pqk+}_{t}]_{\mathcal S(i,j)}\right\|+(\zeta_{1}-\delta_{1})\left\|[\bm x_{ij,1}^{pq}]_{k}\right\|
\\&\leq \zeta_{1}\Lambda<\Lambda.
\end{align*}

One can deduce the rest from this and $0<\zeta_{t}<1$, then $\left\|[\bm x_{ij,t}^{pq}]_{k}\right\|<\Lambda$. Overall, $\left\|\bm x_{ij,t}\right\|< nd\Lambda$ for any $t\in[T]$, $i\in[N]$, and $j\in\mathcal H_{i}$.

From (\ref{eq.1}) and (\ref{eq.2}), one has
\begin{align*}
\left\|\lambda_{ij,t+1}\right\|&\leq (1-\beta_{t}\eta_{t})\left\|\lambda_{ij,t}\right\|+\eta_{t}\left\|y_{ij,t}\right\|.
\end{align*}

By $\left\|\lambda_{ij,1}\right\|=0$,
\begin{equation}\label{eq.0000}
\begin{aligned}
\left\|\lambda_{ij,t+1}\right\|&\leq \sum_{v=0}^{t-1}(1-\beta_{v}\eta_{v})_{[t:t-v+1]}\eta_{t-v}\left\|y_{ij,t-v}\right\|,
\end{aligned}
\end{equation}
where $(1-\beta_{v}\eta_{v})_{[t:t-v+1]}=(1-\beta_{t}\eta_{t})\times \cdots \times(1-\beta_{t-v+1}\eta_{t-v+1})$ and $(1-\beta_{v}\eta_{v})_{[t:t+1]}=1$.

Summing (\ref{eq.0000}) over $t=1,\cdots,T-1$ and taking expectation in $\mathcal U_{t}$, one gets
\begin{equation}\label{eq.04}
\begin{aligned}
\sum_{t=1}^{T}\left\|\lambda_{ij,t}\right\|&\leq \eta_{T}T\sum_{t=1}^{T}\mathbf E_{\mathcal U_t}\left\|y_{ij,t}\right\|
\\&\leq \mathcal O(\eta_{T}T^{2}),
\end{aligned}
\end{equation}
where the first inequality holds since $\lambda_{ij,t}$ is independent of $\mathcal U_t$ and the last inequality results from Assumptions \ref{assp.111} and \ref{assp.7}, and the fact that
\begin{equation}\label{eq.1122}
\begin{aligned}
\mathbf E_{\mathcal U_t}\left\|y_{ij,t}\right\|&\leq \mathbf E_{\mathcal U_t}\left\|g_{t}(\bm x_{ij,t},\omega_{t})-g_{t}(\bm \tau_{ij,t},\omega_{t})+G_{t}(\bm \tau_{ij,t})\right\|
\\&\leq l_{g_{1}}\left\|\bm x_{ij,t}-\bm\tau_{ij,t}\right\|+\left\|G_{t}(\bm \tau_{ij,t})\right\|
\\&\leq 2l_{g_{1}}nd\Lambda+B_{1}.
\end{aligned}
\end{equation}

From the update rule (\ref{eq.2}), for any $\lambda \in \mathbb R^{m}_{\geq 0}$, one obtains
\begin{equation}\label{eq.71}
\begin{aligned}
&\left\|\lambda_{ij,t+1}-\lambda\right\|^{2}
\\&=\left\|[\lambda_{ij,t}+\eta_{t}(y_{ij,t}-\beta_{t}\lambda_{ij,t})]_{+}-\lambda\right\|^{2}
\\&\leq \left\|\lambda_{ij,t}-\lambda\right\|^{2}+2\left<\lambda_{ij,t},\eta_{t}(y_{ij,t}-\beta_{t}\lambda_{ij,t})\right>
\\&\quad-2\left<\lambda,\eta_{t}(y_{ij,t}-\beta_{t}\lambda_{ij,t})\right>+\eta_{t}^{2}\left\|y_{ij,t}-\beta_{t}\lambda_{ij,t}\right\|^{2}
\\&\leq\left\|\lambda_{ij,t}-\lambda\right\|^{2}+2\eta_{t}\left\|\lambda_{ij,t}\right\|(\left\|y_{ij,t}\right\|
-\beta_{t}\left\|\lambda_{ij,t}\right\|)
\\&\quad-2\eta_{t}\left<\lambda,y_{ij,t}\right>+\eta_{t}\beta_{t}\left\|\lambda\right\|^{2}+\eta_{t}\beta_{t}\left\|\lambda_{ij,t}\right\|^{2}
\\&\quad+2\beta_{t}^{2}\eta_{t}^{2}\left\|\lambda_{ij,t}\right\|^{2}+2\eta_{t}^{2}\left\|y_{ij,t}\right\|^{2}
\\&\leq\left\|\lambda_{ij,t}-\lambda\right\|^{2}+2\eta_{t}\left\|\lambda_{ij,t}\right\|\left\|y_{ij,t}\right\|
-2\eta_{t}\left<\lambda,y_{ij,t}\right>
\\&\quad+\eta_{t}\beta_{t}\left\|\lambda\right\|^{2}+2\eta_{t}^{2}\left\|y_{ij,t}\right\|^{2},
\end{aligned}
\end{equation}
where the first inequality results from the nonexpansive property of the operator $[\cdot]_{+}$, and the last inequality results from dropping the negative term and $0<\beta_{t}\eta_{t}<\frac{1}{2}$.

Regrouping the terms on (\ref{eq.71}), one has
\begin{equation}\label{eq.72}
\begin{aligned}
&2\eta_{t}\left<\lambda,y_{ij,t}\right>-\eta_{t}\beta_{t}\left\|\lambda\right\|^{2}
\\&\leq\left\|\lambda_{ij,t}-\lambda\right\|^{2}-\left\|\lambda_{ij,t+1}-\lambda\right\|^{2}+2\eta_{t}\left\|\lambda_{ij,t}\right\|\left\|y_{ij,t}\right\|
\\&\quad+2\eta_{t}^{2}\left\|y_{ij,t}\right\|^{2}.
\end{aligned}
\end{equation}

Furthermore, by Cauchy-Schwarz inequality, one obtains
\begin{equation}\label{eq.73}
\begin{aligned}
&2\eta_{t}\left<\lambda,y_{ij,t}-G_{t}(x^{*}_{\mathcal H,t})\right>\geq -2\eta_{t}\left\|\lambda\right\|\left\|y_{ij,t}-G_{t}(x^{*}_{\mathcal H,t})\right\|.
\end{aligned}
\end{equation}

Combining (\ref{eq.73}) with (\ref{eq.72}), one reaches
\begin{equation}\label{eq.74}
\begin{aligned}
&2\eta_{t}\left<\lambda,G_{t}(x^{*}_{\mathcal H,t})\right>-\eta_{t}\beta_{t}\left\|\lambda\right\|^{2}\quad\quad\quad\quad\quad\quad\quad\quad\quad\quad\quad\quad
\\&\leq\left\|\lambda_{ij,t}-\lambda\right\|^{2}-\left\|\lambda_{ij,t+1}-\lambda\right\|^{2}+2\eta_{t}\left\|\lambda_{ij,t}\right\|\left\|y_{ij,t}\right\|
\\&\quad+2\eta_{t}^{2}\left\|y_{ij,t}\right\|^{2}+2\eta_{t}\left\|\lambda\right\|\left\|y_{ij,t}-G_{t}(x^{*}_{\mathcal H,t})\right\|.
\end{aligned}
\end{equation}

Submitting $\lambda=\frac{[G_{t}(x^{*}_{\mathcal H,t})]_{+}}{\beta_{t}}$ into (\ref{eq.74}) yields
\begin{equation}\label{eq.75}
\begin{aligned}
&\left\|[G_{t}(x^{*}_{\mathcal H,t})]_{+}\right\|^{2}
\\&\leq\frac{\beta_{t}}{\eta_{t}}\left(\left\|\lambda_{ij,t}-\lambda\right\|^{2}-\left\|\lambda_{ij,t+1}-\lambda\right\|^{2}\right)+2\beta_{t}\left\|\lambda_{ij,t}\right\|\left\|y_{ij,t}\right\|
\\&\quad+2\eta_{t}\beta_{t}\left\|y_{ij,t}\right\|^{2}+2\left\|[G_{t}(x^{*}_{\mathcal H,t})]_{+}\right\|\left\|y_{ij,t}-G_{t}(x^{*}_{\mathcal H,t})\right\|.
\end{aligned}
\end{equation}

Summing (\ref{eq.75}) over $i=1,\cdots,N$ and taking expectation in $\mathcal U_{t}$ on both sides of (\ref{eq.75}) lead to
\begin{equation}\label{eq.76}
\begin{aligned}
&\sum_{t=1}^{T}\mathbf{E}_{\mathcal U_t}\left\|[G_{t}(x^{*}_{\mathcal H,t})]_{+}\right\|^{2}
\\&\leq2(2l_{g_{1}}nd\Lambda+B_{1})\beta_{T}\sum_{t=1}^{T}\left\|\lambda_{ij,t}\right\|\\&\quad+18B_1^{2}\beta_{T}\eta_{T}T
+2B_{1}\sum_{t=1}^{T}\mathbf E_{\mathcal U_t}\big\|y_{ij,t}
-G_{t}(x^{*}_{\mathcal H,t})\big\|,
\end{aligned}
\end{equation}
where  the last inequality follows from that $\lambda_{ij,1}=\bm 0_{m}$, $G_{1}(x^{*}_{\mathcal H,1})\leq \bm 0_{m}$, the property of $\beta_{t}$ and $\eta_{t}$, (\ref{eq.1122}), and Assumption \ref{assp.111}.

Next, by the update rule (\ref{eq.1}), one has
\begin{equation}\label{eq.780}
\begin{aligned}
&\sum_{t=1}^{T}\mathbf E_{\mathcal U_t}\left\|y_{ij,t}-G_{t}(x^{*}_{\mathcal H,t})\right\|
\\&\leq 2\sum_{t=1}^{T}\mathbf E_{\mathcal U_t}\left\|g_{t}(\bm x_{ij,t},\omega_{t})-g_{t}(\bm \tau_{ij,t},\omega_{t})+G_{t}(\bm \tau_{ij,t})\right.
\\&\quad\left.-G_{t}(\bm x_{ij,t})\right\|+2\sum_{t=1}^{T}\mathbf E_{\mathcal U_t}\left\|G_{t}(\bm x_{ij,t})-G_{t}(x_{\mathcal H,t}^{*})\right\|
\\&\leq 4l_{g_{1}}\sum_{t=1}^{T}\left\|\bm x_{ij,t}-\bm \tau_{ij,t}\right\|+2l_{g_{1}}\sum_{t=1}^{T}\left\|\bm x_{ij,t}- x_{\mathcal H,t}^{*}\right\|
\\&\leq \mathcal O(1+\sum_{t=1}^{T}\alpha_{t}+\sum_{t=1}^{T}\xi_{t}+\Phi(T)),\quad\quad\quad\quad\quad\quad\quad\quad\quad\quad\quad\quad\quad\quad\quad
\end{aligned}
\end{equation}
where the second inequality results from Assumption \ref{assp.7}, and the last inequality follows from a similar procedure in (\ref{eq.1118}).

Substituting (\ref{eq.780}) into (\ref{eq.76}) and taking full expectation, by (\ref{eq.04}) and $\sum_{t=1}^{T}\mathbf{E}\left\|a_{t}\right\|\leq\sqrt{T\sum_{t=1}^{T}\mathbf{E}\left\|a_{t}\right\|^{2}}$, the result is obtained. \hfill$\blacksquare$

\subsection{Proof of Theorem \ref{Theorem_2}}\label{prof_lemma_8}

Let $\nabla_{x_{ij,t}}F_{i,t}(\cdot)\triangleq \nabla F_{i,t}(\cdot)$, $\nabla_{x_{ij,t}}G_{t}(\cdot)\triangleq \nabla G_{t}(\cdot)$, and $\sum_{t,i,j,k}(\cdot)\triangleq \sum_{t=1}^{T}\sum_{i=1}^{N}\sum_{j\in\mathcal H_{i}}\sum_{k=1}^{d}(\cdot)$. By the boundedness of $x_{ij,t}$ and $x^{*}_{ij,t}$, one has
\begin{equation}\label{eq.81}
\begin{aligned}
&\mathbf E_{\mathcal U_t}\left\|x_{ij,t+1}-x^{*}_{ij,t+1}\right\|^{2}
\\&=\mathbf E_{\mathcal U_t}\left\|x_{ij,t+1}-x^{*}_{ij,t}+x^{*}_{ij,t}-x^{*}_{ij,t+1}\right\|^{2}
\\&=\mathbf E_{\mathcal U_t}\left\|x_{ij,t+1}-x^{*}_{ij,t}\right\|^{2}+\left\|x^{*}_{ij,t}-x^{*}_{ij,t+1}\right\|^{2}
\\&\quad+2\mathbf E_{\mathcal U_t}\left<x_{ij,t+1}-x^{*}_{ij,t},x^{*}_{ij,t}-x^{*}_{ij,t+1}\right>
\\&\leq \mathbf E_{\mathcal U_t}\left\|x_{ij,t+1}-x^{*}_{ij,t}\right\|^{2}+4R\left\|x^{*}_{ij,t}-x^{*}_{ij,t+1}\right\|.
\end{aligned}
\end{equation}

For the first term on the right hand-side of (\ref{eq.81}), one obtains
\begin{equation}\label{eq.82}
\mathbf E_{\mathcal U_t}\left\|x_{ij,t+1}-x^{*}_{ij,t}\right\|^{2}=\sum_{k=1}^{d}\mathbf E_{\mathcal U_t}\left\|[x_{ij,t+1}]_{k}-[x^{*}_{ij,t}]_{k}\right\|^{2}.
\end{equation}

Furthermore, an upper bound of $\mathbf E_{\mathcal U_t}\left\|[x_{ij,t+1}]_{k}-[x^{*}_{ij,t}]_{k}\right\|^{2}$ can be constructed as
\begin{equation}\label{eq.83}
\begin{aligned}
&\mathbf E_{\mathcal U_t}\Big\|[x_{ij,t+1}]_{k}-[x^{*}_{ij,t}]_{k}\Big\|^{2}
\\&\overset{(a)}{\leq} \mathbf E_{\mathcal U_t}\bigg\|(1-\alpha_{t})([x_{ij,t}]_{k}-[x^{*}_{ij,t}]_{k})
\\&\quad+\alpha_{t}P_{\Omega_{ij}^{k}}\Big([x_{ij,t}]_{k}-\gamma_{t}[v^{j}_{ij,t}]_{k}-\gamma_{t}[\nabla_{x_{ij,t}}y_{ij,t}^\top]_{(k,:)}\lambda_{ij,t+1}\Big)
\\&\quad-\alpha_{t}P_{\Omega_{ij}^{k}}\Big([x_{ij,t}^{*}]_{k}-\gamma_{t}[\nabla F_{i,t}(x^{*}_{t})]_{k}-\gamma_{t}[\nabla G_{t}(x^{*}_{t})^\top]_{(k,:)}\lambda_{t}^{*}\Big)\bigg\|^{2}
\\&\overset{(b)}{\leq} (1-\alpha_{t})\Big\|[x_{ij,t}]_{k}-[x^{*}_{ij,t}]_{k}\Big\|^{2}+\alpha_{t}\mathbf E_{\mathcal U_t}\bigg\|[x_{ij,t}]_{k}-[x^{*}_{ij,t}]_{k}
\\&\quad+\gamma_{t}[\nabla G_{t}(x^{*}_{t})^\top]_{(k,:)}\lambda_{t}^{*}-\gamma_{t}\Big([\nabla F_{i,t}(x_{\mathcal H,t}^{*})]_{k}-[\nabla F_{i,t}(x^{*}_{t})]_{k}\Big)
\\&\quad-\gamma_{t}\Big([v^{j}_{ij,t}]_{k}-[\nabla F_{i,t}(x_{\mathcal H,t}^{*})]_{k}\Big)-\gamma_{t}[\nabla_{x_{ij,t}}y_{ij,t}^\top]_{(k,:)}\lambda_{ij,t+1})\bigg\|^{2}
\\&\leq\Big\|[x_{ij,t}]_{k}-[x^{*}_{ij,t}]_{k}\Big\|^{2}
\\&\quad+2\alpha_{t}\gamma_{t}\Big<[x_{ij,t}]_{k}-[x^{*}_{ij,t}]_{k},[\nabla G_{t}(x^{*}_{t})^\top]_{(k,:)}\lambda_{t}^{*}\Big>
\\&\quad-2\alpha_{t}\gamma_{t}\Big<[x_{ij,t}]_{k}-[x^{*}_{ij,t}]_{k},[\nabla F_{i,t}(x_{\mathcal H,t}^{*})]_{k}-[\nabla F_{i,t}(x^{*}_{t})]_{k}\Big>
\\&\quad-2\alpha_{t}\gamma_{t}\mathbf E_{\mathcal U_t}\Big<[x_{ij,t}]_{k}-[x^{*}_{ij,t}]_{k},[v^{j}_{ij,t}]_{k}-[\nabla F_{i,t}(x_{\mathcal H,t}^{*})]_{k}\Big>
\\&\quad-2\alpha_{t}\gamma_{t}\mathbf E_{\mathcal U_t}\Big<[x_{ij,t}]_{k}-[x^{*}_{ij,t}]_{k},[(\nabla_{x_{ij,t}}y_{ij,t})^\top]_{(k,:)}\lambda_{ij,t+1})\Big>
\\&\quad+2\alpha_{t}\gamma_{t}^{2}\mathbf E_{\mathcal U_t}\Big\|[\nabla G_{t}(x^{*}_{t})^\top]_{(k,:)}\lambda_{t}^{*}-[\nabla_{x_{ij,t}}y_{ij,t}^\top]_{(k,:)}\lambda_{ij,t+1}\Big\|^{2}
\\&\quad+4\alpha_{t}\gamma_{t}^{2}\Big\|[\nabla F_{i,t}(x^{*}_{t})]_{k}\Big\|^{2}+4\alpha_{t}\gamma_{t}^{2}\mathbf E_{\mathcal U_t}\left\|[v^{j}_{ij,t}]_{k}\right\|^{2}
\\&\overset{(c)}{\leq}\Big\|[x_{ij,t}]_{k}-[x^{*}_{ij,t}]_{k}\Big\|^{2}
\\&\quad+2\alpha_{t}\gamma_{t}\Big<[x_{ij,t}]_{k}-[x^{*}_{ij,t}]_{k},[\nabla G_{t}(x^{*}_{t})^\top]_{(k,:)}\lambda_{t}^{*}\Big>
\\&\quad-2\alpha_{t}\gamma_{t}\Big<[x_{ij,t}]_{k}-[x^{*}_{ij,t}]_{k},[\nabla F_{i,t}(x_{\mathcal H,t}^{*})]_{k}-[\nabla F_{i,t}(x^{*}_{t})]_{k}\Big>
\\&\quad+4\alpha_{t}\gamma_{t}R\mathbf E_{\mathcal U_t}\Big\|[v^{j}_{ij,t}]_{k}-[\nabla F_{i,t}(x_{\mathcal H,t}^{*})]_{k}\Big\|
\\&\quad+4\alpha_{t}\gamma_{t}^{2}\mathbf E_{\mathcal U_t}\left\|[v^{j}_{ij,t}]_{k}\right\|^{2}+(4\vartheta^{2}B_{2}^{2}+4B_{4}^{2})\alpha_{t}\gamma_{t}^{2}
\\&\quad+4\alpha_{t}\gamma_{t}R\mathbf E_{\mathcal U_t}\Big(\left\|[\nabla_{x_{ij,t}}y_{ij,t}^\top]_{(k,:)}\right\|\left\|\lambda_{ij,t+1}\right\|\Big)
\\&\quad+4\alpha_{t}\gamma_{t}^{2}\mathbf E_{\mathcal U_t}\Big(\left\|[\nabla_{x_{ij,t}}y_{ij,t}^\top]_{(k,:)}\right\|^{2}\left\|\lambda_{ij,t+1}\right\|^{2}\Big),\quad\quad\quad\quad\quad\quad
\end{aligned}
\end{equation}
where $\Omega_{ij}^{k}$ represents the $k$-th coordinate of $\Omega_{ij}$, $(a)$ results from (\ref{eq.3}) and the optimal condition (\ref{re_3}); $(b)$ follows from Jensen's inequality and the nonexpansive property of the operator $P_{\Omega_{ij}^{k}}(\cdot)$; and $(c)$ results from Assumption \ref{assp.111}, and the fact that there exists a constant $\vartheta$, s.t., $\left\|\lambda_{t}^{*}\right\|\leq \vartheta$ for any $t\in[T]$ due to Lemma 1 in \cite{Nedic_2009}.

Substituting (\ref{eq.83}) into (\ref{eq.82}) and (\ref{eq.81}), rearranging the terms, dividing both sides by $\frac{1}{2\alpha_{t}\gamma_{t}}$ and summing it over $i\in[N]$, $j\in\mathcal H_{i}$ and $t\in[T]$, then one has
\begin{equation}\label{eq.84}
\begin{aligned}
&\underbrace{\sum_{t,i,j,k}\big<[x_{ij,t}]_{k}-[x^{*}_{ij,t}]_{k},[\nabla F_{i,t}(x_{\mathcal H,t}^{*})]_{k}-[\nabla F_{i,t}(x^{*}_{t})]_{k}\big>}_{=\colon S_{1}}
\\&\leq\underbrace{ \frac{1}{2}\sum_{t,i,j}\frac{1}{\alpha_{t}\gamma_{t}}\left(\left\|x_{ij,t}-x^{*}_{ij,t}\right\|^{2}-\mathbf E_{\mathcal U_t}\left\|x_{ij,t+1}-x^{*}_{ij,t+1}\right\|^{2}\right)}_{=\colon S_{2}}
\\&\quad+\underbrace{\sum_{t,i,j,k}\big<[x_{ij,t}]_{k}-[x^{*}_{ij,t}]_{k},[\nabla G_{t}(x_{t}^{*})^\top]_{(k,:)}\lambda_{t}^{*}\big>}_{=\colon S_{3}}
\\&\quad+\underbrace{2R\sum_{t,i,j,k}\mathbf E_{\mathcal U_t}\left\|[v^{j}_{ij,t}]_{k}-[\nabla F_{i,t}(x_{\mathcal H,t}^{*})]_{k}\right\|}_{=\colon S_{4}}+\frac{2R\Phi_{\mathcal H}(T)}{\alpha_{t_{1}}\gamma_{t_{1}}}
\\&\quad+\underbrace{2\sum_{t,i,j,k}\gamma_{t}\mathbf E_{\mathcal U_t}\left\|[v^{j}_{ij,t}]_{k}\right\|^{2}}_{=\colon S_{5}}+(2d\vartheta^{2}B_{2}^{2}+2dB_{4}^{2})|\mathcal H|\sum_{t=1}^{T}\gamma_{t}
\\&\quad+\underbrace{2dC_{5}R\sum_{t,i,j}\mathbf E_{\mathcal U_t}\left\|\lambda_{ij,t+1}\right\|+6dB_{2}^{2}\sum_{t,i,j}\gamma_{t}\mathbf E_{\mathcal U_t}\left\|\lambda_{ij,t+1}\right\|^{2}}_{=\colon S_{6}}.
\end{aligned}
\end{equation}

For term $S_{1}$, under Assumption \ref{assp.8}, one obtains
\begin{equation}\label{eq.85}
\begin{aligned}
S_{1}&=\sum_{t,i,j,k}\big<[x_{ij,t}]_{k}-[x^{*}_{ij,t}]_{k},[\nabla F_{i,t}(x_{\mathcal H,t}^{*})]_{k}-[\nabla F_{i,t}(x^{*}_{t})]_{k}\big>
\\&\geq \sigma\sum_{t=1}^{T}\sum_{i=1}^{N}\sum_{j\in\mathcal H_{i}}\left\|x_{ij,t}-x_{ij,t}^{*}\right\|^{2}.\quad\quad\quad\quad\quad\quad\quad\quad
\end{aligned}
\end{equation}

Taking expectation with respect to $\mathscr U_{T}$ on term $S_{2}$ and due to the properties of $\left\{\alpha_{t}\right\}_{t=1}^{T}$ and $\left\{\gamma_{t}\right\}_{t=1}^{T}$, one has
\begin{align*}
&\frac{1}{2}\sum_{t=1}^{T}\frac{1}{\alpha_{t}\gamma_{t}}\left(\mathbf E\left\|x_{ij,t}-x^{*}_{ij,t}\right\|^{2}-\mathbf E\left\|x_{ij,t+1}-x^{*}_{ij,t+1}\right\|^{2}\right)
\\&\overset{(a)}{\leq} \frac{1}{2}\sum_{t=1}^{t_{1}}\frac{1}{\alpha_{t_{1}}\gamma_{t_{1}}}\mathbf E\left(\left\|x_{ij,t}-x^{*}_{ij,t}\right\|^{2}-\left\|x_{ij,t+1}-x^{*}_{ij,t+1}\right\|^{2}\right)
\\&\quad +\frac{1}{2}\sum_{t=t_{1}+1}^{T}\frac{1}{\alpha_{t}\gamma_{t}}\mathbf E\left(\left\|x_{ij,t}-x^{*}_{ij,t}\right\|^{2}-\left\|x_{ij,t+1}-x^{*}_{ij,t+1}\right\|^{2}\right)
\\&\overset{(b)}{\leq} \frac{1}{2\alpha_{t_{1}}\gamma_{t_{1}}}\mathbf E\left\|x_{ij,1}-x_{ij,1}^{*}\right\|^{2}
\\&\quad+\left(\frac{1}{2\alpha_{t_{1}+1}\gamma_{t_{1}+1}}-\frac{1}{2\alpha_{t_{1}}\gamma_{t_{1}}}\right)\mathbf E\left\|x_{ij,t_{1}+1}-x_{ij,t_{1}+1}^{*}\right\|^{2}
\\&\overset{(c)}{\leq} 2R^{2}\alpha_{t_{1}}^{-1}\gamma_{t_{1}}^{-1},
\end{align*}
where $(a)$ results from that $\alpha_{t_{1}}\gamma_{t_{1}}\leq\alpha_{t}\gamma_{t}$ for $\forall t \in [1,t_{1}]$, $(b)$ follows from $\alpha_{t}$, $\gamma_{t}$ are time-invariant for $\forall t>t_{1}$ and dropping the negative term, and $(c)$ results from the boundedness of $x_{ij,1}$ and $x_{ij,1}^{*}$ and the fact that $\frac{1}{\alpha_{t_{1}+1}\gamma_{t_{1}+1}}-\frac{1}{\alpha_{t_{1}}\gamma_{t_{1}}}\leq0$.

Therefore, for term $S_{2}$, one obtains
\begin{equation}\label{eq.86}
\begin{aligned}
&\mathbf E_{\mathscr{U}_T}(S_{2})
\\&=\frac{1}{2}\sum_{t,i,j}\frac{1}{\alpha_{t}\gamma_{t}}\left(\mathbf E\left\|x_{ij,t}-x^{*}_{ij,t}\right\|^{2}-\mathbf E\left\|x_{ij,t+1}-x^{*}_{ij,t+1}\right\|^{2}\right)
\\&\leq 2|\mathcal H|R^{2}\alpha_{t_{1}}^{-1}\gamma_{t_{1}}^{-1}.\quad\quad\quad\quad\quad\quad\quad\quad\quad\quad\quad\quad\quad\quad\quad\quad
\end{aligned}
\end{equation}

For term $S_{3}$, one has
\begin{equation}\label{eq.87}
\begin{aligned}
S_{3}=&\sum_{t,i,j,k}\big<[x_{ij,t}]_{k}-[x^{*}_{ij,t}]_{k},[\nabla G_{t}(x_{t}^{*})^\top]_{(k,:)}\lambda_{t}^{*}\big>
\\&\overset{(a)}{\leq}\sum_{t=1}^{T}\lambda_{t}^{*\top}\left(G_{t}(x^{*}_{\mathcal H,t})-G_{t}(x_{t}^{*})\right)
\\&\overset{(b)}{\leq}\sum_{t=1}^{T}\lambda_{t}^{*\top}[G_{t}(x^{*}_{\mathcal H,t})]_{+}
\\&\overset{(c)}{\leq}\vartheta\mathcal{CV}_{\mathcal H}(T),
\end{aligned}
\end{equation}
where $(a)$ results from the convexity of $G_{t}(x_{t})$, $(b)$ follows from the complementary slackness condition $\lambda_{t}^{*\top}G_{t}(x_{t}^{*})=0$ originating from (\ref{re_2}) and the nonnegativity of $[G_{t}(x^{*}_{\mathcal H,t})]_{+}$, and $(c)$ results from the boundedness of $\left\|\lambda_{t}^{*}\right\|$.

For term $S_{4}$, one obtains
\begin{equation}\label{eq.88}
\begin{aligned}
S_{4}&=2R\sum_{t,i,j,k}\mathbf E_{\mathcal U_t}\left\|[v^{j}_{ij,t}]_{k}-[\nabla F_{i,t}(x_{\mathcal H,t}^{*})]_{k}\right\|
\\&\leq2R\sum_{t,i,j,k}\mathbf E_{\mathcal U_t}\Big\|[v^{j}_{ij,t}]_{k}-[\bar\epsilon_{ij,t}]_{k}\Big\|
\\&\quad +2R\sum_{t,i,j,k}\mathbf E_{\mathcal U_t}\Big\|[\bar \epsilon_{ij,t}]_{k}-[\nabla F_{i,t}(x_{\mathcal H,t}^{*})]_{k}\Big\|.
\end{aligned}
\end{equation}

For the first term of (\ref{eq.88}), by Lemma \ref{lemma_3}, one has
\begin{equation}\label{eq.89}
\begin{aligned}
&2R\sum_{t,i,j,k}\mathbf E_{\mathcal U_t}\Big\|[v^{j}_{ij,t}]_{k}-[\bar \epsilon_{ij,t}]_{k}\Big\|
\\&\leq \mathcal O(1+\sum_{t=1}^{T}\alpha_{t}+\sum_{t=1}^{T}\xi_{t}+\Phi(T)+\Delta F^{\text{sup}}_{T}).
\end{aligned}
\end{equation}

For simplification, some abbreviations are made as follows.
\begin{align*}
&\nabla\widetilde F_{i,t}(\bm x_{\mathcal H_{i},t})\triangleq\frac{1}{|\mathcal H_{i}|}\sum_{q\in\mathcal H_{i}}\nabla_{x_{ij,t}}f^{i,t}_{q}(\bm x_{iq,t},\theta_{i,t}), 
\\&\nabla\widetilde F_{i,t}(\bm \tau_{\mathcal H_{i},t})\triangleq\frac{1}{|\mathcal H_{i}|}\sum_{q\in\mathcal H_{i}}\nabla_{x_{ij,t}}f^{i,t}_{q}(\bm \tau_{iq,t},\theta_{i,t}),
\\&\nabla\widetilde F_{i,t}(x_{\mathcal H,t}^{*})\triangleq\frac{1}{|\mathcal H_{i}|}\sum_{q\in\mathcal H_{i}}\nabla_{x_{ij,t}}f^{i,t}_{q}(x_{\mathcal H,t}^{*},\theta_{i,t}),
\\&\nabla F_{i,t}(\bm \tau_{\mathcal H_{i},t})\triangleq\frac{1}{|\mathcal H_{i}|}\sum_{q\in\mathcal H_{i}}\nabla_{x_{ij,t}}\mathbf E_{\theta_{i,t}}\left[f^{i,t}_{q}(\bm \tau_{iq,t},\theta_{i,t})\right]. 
\end{align*}

Following these notations, one obtains
\begin{align*}
&\mathbf E_{\mathcal U_t}\Big\|\bar \epsilon_{ij,t}-\nabla F_{i,t}(x_{\mathcal H,t}^{*})\Big\|^{2}
\\&\overset{(a)}{=}\mathbf E_{\mathcal U_t}\Big\|\nabla\widetilde F_{i,t}(\bm x_{\mathcal H_{i},t})-\nabla\widetilde F_{i,t}(x_{\mathcal H,t}^{*})+\nabla\widetilde F_{i,t}(x_{\mathcal H,t}^{*})
\\&\quad\quad~-\nabla\widetilde F_{i,t}(\bm \tau_{\mathcal H_{i},t})+\nabla F_{i,t}(\bm \tau_{\mathcal H_{i},t})-\nabla F_{i,t}(x_{\mathcal H,t}^{*})\Big\|^{2}
\\&\leq 2\mathbf E_{\mathcal U_t}\Big\|\nabla\widetilde F_{i,t}(\bm x_{\mathcal H_{i},t})-\nabla\widetilde F_{i,t}(x_{\mathcal H,t}^{*})\Big\|^{2}+2\mathbf E_{\mathcal U_t}\Big\|\nabla\widetilde F_{i,t}(x_{\mathcal H,t}^{*})
\\&\quad\quad~-\nabla\widetilde F_{i,t}(\bm \tau_{\mathcal H_{i},t})+\nabla F_{i,t}(\bm \tau_{\mathcal H_{i},t})-\nabla F_{i,t}(x_{\mathcal H,t}^{*})\Big\|^{2}
\\&\overset{(b)}{\leq} \frac{2l_{f}^{2}}{|\mathcal H_{i}|}\sum_{j \in \mathcal H_{i}}\left\|\bm x_{ij,t}^{\mathcal H}-x_{t}^{\mathcal H}\right\|^{2}+\frac{2l_{f}^{2}}{|\mathcal H_{i}|}\sum_{j \in \mathcal H_{i}}\left\|\bm x_{ij,t}^{\mathcal B}-x_{t}^{\mathcal B*}\right\|^{2}
\\&\quad+2\mathbf E_{\mathcal U_t}\Big\|\nabla\widetilde F_{i,t}(x_{\mathcal H,t}^{*})-\nabla\widetilde F_{i,t}(\bm\tau_{\mathcal H_{i},t})\Big\|^{2}
\\&\quad-2\Big\|\nabla F_{i,t}(\bm \tau_{\mathcal H_{i},t})-\nabla F_{i,t}(x_{\mathcal H,t}^{*})\Big\|^{2}
\\&\overset{(c)}{\leq}\frac{2l_{f}^{2}}{|\mathcal H_{i}|}\sum_{j \in \mathcal H_{i}}\left\|\bm x_{ij,t}^{\mathcal H}-x_{t}^{\mathcal H}\right\|^{2}+\frac{2l_{f}^{2}}{|\mathcal H_{i}|}\sum_{j \in \mathcal H_{i}}\left\|\bm x_{ij,t}^{\mathcal B}-x_{t}^{\mathcal B*}\right\|^{2}
\\&\quad +\frac{2l_{f}^{2}}{|\mathcal H_{i}|}\sum_{j \in \mathcal H_{i}}\left\|\bm \tau_{ij,t}^{\mathcal H}-x_{t}^{\mathcal H}\right\|^{2}+\frac{2l_{f}^{2}}{|\mathcal H_{i}|}\sum_{j \in \mathcal H_{i}}\left\|\bm \tau_{ij,t}^{\mathcal B}-x_{t}^{\mathcal B*}\right\|^{2},
\end{align*}
with $x_{t}^{\mathcal H}=col(x_{ij,t})_{i\in[N],j\in\mathcal H_{i}}$, $x_{t}^{\mathcal B*}=col(x_{ij,t}^{*})_{i\in[N],j\in\mathcal B_{i}}$, $\bm x_{ij,t}^{\mathcal H}=col(\bm x_{ij,t}^{pq})_{p\in[N],q\in\mathcal H_{p}}$, $\bm x_{ij,t}^{\mathcal B}=col(\bm x_{ij,t}^{pq})_{p\in[N],q\in\mathcal B_{p}}$, $\bm \tau_{ij,t}^{\mathcal H}=col(\bm \tau_{ij,t}^{pq})_{p\in[N],q\in\mathcal H_{p}}$, $\bm \tau_{ij,t}^{\mathcal B}=col(\bm \tau_{ij,t}^{pq})_{p\in[N],q\in\mathcal B_{p}}$, where $(a)$ results from the definition of $\bar \epsilon_{ij,t}$, $(b)$ follows from Assumption \ref{assp.7} and because $\tau_{iq,t}$ is independent with $\mathcal U_t$, one can derive from Assumption \ref{assp.11} that $\mathbf E_{\mathcal U_t}[\nabla_{x_{ij,t}}f^{i,t}_{q}(x^{*}_{\mathcal H,t},\theta_{i,t})]=\nabla_{x_{ij,t}}\mathbf E_{\theta_{i,t}}[f^{i,t}_{q}(x^{*}_{\mathcal H,t},\theta_{i,t})]$ and $\mathbf E_{\mathcal U_t}[\nabla_{x_{ij,t}}f^{i,t}_{q}(\bm \tau_{iq,t},\theta_{i,t})]=\nabla_{x_{ij,t}}\mathbf E_{\theta_{i,t}}[f^{i,t}_{q}(\bm \tau_{iq,t},\theta_{i,t})]$, and $(c)$ results from dropping the negative term.

Thus, for the second term of (\ref{eq.88}), one has
\begin{equation}\label{eq.810}
\begin{aligned}
&2R\sum_{t,i,j,k}\mathbf E_{\mathcal U_t}\Big\|[\bar \epsilon_{ij,t}]_{k}-[\nabla F_{i,t}(x_{\mathcal H,t}^{*})]_{k}\Big\|
\\&\leq2R\sum_{t,i,j,k}\mathbf E_{\mathcal U_t}\Big\|\bar \epsilon_{ij,t}-\nabla F_{i,t}(x_{\mathcal H,t}^{*})\Big\|
\\&\overset{(a)}{\leq}2R\sum_{t,i,j,k}\sqrt{\mathbf E_{\mathcal U_t}\Big\|\bar \epsilon_{ij,t}-\nabla F_{i,t}(x_{\mathcal H,t}^{*})\Big\|^{2}}
\\&\leq 2\sqrt{2}dl_{f}\sqrt{|\mathcal H|}R\Big(\sum_{t,i,j}\left\|\bm x_{ij,t}^{\mathcal H}-x_{t}^{\mathcal H}\right\|+\sum_{t,i,j}\left\|\bm \tau_{ij,t}^{\mathcal H}-x_{t}^{\mathcal H}\right\|\Big)
\\&\quad+2\sqrt{2}dl_{f}\sqrt{|\mathcal H|}R\Big(\sum_{t,i,j}\left\|\bm x_{ij,t}^{\mathcal B}-x_{t}^{\mathcal B*}\right\|+\sum_{t,i,j}\left\|\bm \tau_{ij,t}^{\mathcal B}-x_{t}^{\mathcal B*}\right\|\Big)
\\&\overset{(b)}{\leq} 4\sqrt{2}dl_{f}t_{0}\sqrt{|\mathcal H|}R\sum_{t,i,j}\Big(\left\|\bm x_{ij,t}^{\mathcal H}-x_{t}^{\mathcal H}\right\|+\left\|\bm x_{ij,t}^{\mathcal B}-x_{t}^{\mathcal B*}\right\|\Big)
\\&\quad+2\sqrt{2}dl_{f}t_{0}|\mathcal H|^{2}R\sum_{l=2}^{T}\Big(\left\|x_{l}^{\mathcal H}-x_{l-1}^{\mathcal H}\right\|+\left\|x_{l}^{\mathcal B*}-x_{l-1}^{\mathcal B*}\right\|\Big)
\\&\overset{(c)}{\leq} \mathcal O\left(1+\sum_{t=1}^{T}\alpha_{t}+\sum_{t=1}^{T}\xi_{t}+\Phi_{\mathcal B}(T)\right),\quad\quad\quad\quad\quad\quad\quad
\end{aligned}
\end{equation}
where $(a)$ results from that $(\mathbf E\left\|x\right\|)^{2}\leq\mathbf E\left\|x\right\|^{2}$, $(b)$ follows from the definition of $\bm \tau_{ij,t}$, and $(c)$ results from Lemma \ref{lemma_2}.

Substituting (\ref{eq.89}) and (\ref{eq.810}) into (\ref{eq.88}), one obtains
\begin{equation}\label{eq.811}
\begin{aligned}
S_{4}\leq \mathcal O(1+\sum_{t=1}^{T}\alpha_{t}+\sum_{t=1}^{T}\xi_{t}+\Phi(T)+\Delta F^{\text{sup}}_{T}).
\end{aligned}
\end{equation}

Combining (\ref{eq.13}) and (\ref{eq.16}), by Assumptions \ref{assp.111} and \ref{assp.7}, one has
\begin{align*}
&\mathbf E_{\mathcal U_t}\left\|\bm v_{ij,t}^{k}\right\|^{2}
\\&\leq \mathbf E_{\mathcal U_t}\Big(C\sum_{g=2}^{t}\Theta^{t-g}\left\|\Delta \bm \epsilon^{k}_{ij,g}\right\|+\sum_{g=2}^{t}v_{\mathcal G}^{t-g}\left\|\Delta \bm \epsilon^{k}_{ij,g}\right\|
\\&\quad+\left\|\bm 1_{|\mathcal H_{i}|}[\bar \epsilon_{ij,t}]_{k}\right\|\Big)^{2}
\end{align*}
\vspace{-3em} 
\begin{equation*}
\begin{aligned}
\\&\leq4(12n^{2}d^{2}l_{f}^{2}|\mathcal H|\Lambda^{2}+(\Delta F^{\text{sup}}_{T})^{2})\Big( C^{2}t\sum_{g=2}^{t}\Theta^{2(t-g)}
\\&\quad+t\sum_{g=2}^{t}v_{\mathcal G}^{2(t-g)}+3B_4^{2}|\mathcal H|\Big),
\end{aligned}
\end{equation*}
where the last inequality results from the fact that
\begin{align*}
\mathbf E_{\mathcal U_t}\left\|\Delta \bm \epsilon_{iq,t}\right\|^{2}
&\leq 4l_{f}^{2}\sum_{j\in \mathcal H_{i}}\big(\left\|\bm x_{ij,t}-\bm \tau_{ij,t}\right\|^{2}
\\&\quad+\left\|\bm x_{ij,t-1}-\bm \tau_{ij,t-1}\right\|^{2}+\left\|\bm \tau_{ij,t}-\bm \tau_{ij,t-1}\right\|^{2}\big)
\\&\quad+4(\Delta F^{\text{sup}}_{T})^{2}
\\&\leq 48n^{2}d^{2}l_{f}^{2}|\mathcal H|\Lambda^{2}+4(\Delta F^{\text{sup}}_{T})^{2}.
\end{align*}

Therefore, for term $S_{5}$, one has
\begin{equation}\label{eq.812}
\begin{aligned}
S_{5}&=2\sum_{t,i,j,k}\gamma_{t}\mathbf E_{\mathcal U_t}\left\|[v^{j}_{ij,t}]_{k}\right\|^{2}
\\&\leq\mathcal O(T\sum_{t=1}^{T}\gamma_{t}+T(\Delta F^{\text{sup}}_{T})^{2}\sum_{t=1}^{T}\gamma_{t}).
\end{aligned}
\end{equation}

Taking expectation with respect to $\mathscr U_{T}$ on term $S_{6}$, from (\ref{eq.0000}) and (\ref{eq.04}), one gets
\begin{equation}\label{eq.8133}
\begin{aligned}
\mathbf E_{\mathscr U_{T}}\left[S_{6}\right]&=2dC_{5}R\sum_{t,i,j}\mathbf E\left\|\lambda_{ij,t+1}\right\|+6dB_{2}^{2}\sum_{t,i,j}\gamma_{t}\mathbf E\left\|\lambda_{ij,t+1}\right\|^{2}
\\&\leq \mathcal O(\eta_{T}T^{2}+\eta_{T}^{2}T^{2}\sum_{t=1}^{T}\gamma_{t}).\quad\quad\quad\quad\quad\quad\quad
\end{aligned}
\end{equation}

Combining (\ref{eq.85}), (\ref{eq.86}), (\ref{eq.87}), (\ref{eq.811}), (\ref{eq.812}) and (\ref{eq.8133}) with (\ref{eq.84}) and taking the full expectation on both sides of (\ref{eq.84}), the result (\ref{eq.06}) can be obtained by the following inequality.
\begin{align*}
&\mathbf{E}\left[\mathcal R_{\mathcal H}(T)\right]
\\& \leq l_{f}\sum_{t=1}^{T}\sum_{i=1}^{N}\sum_{j\in\mathcal H_{i}}\mathbf E\left\|x_{ij,t}-x^{*}_{ij,t}\right\|
\\&\leq l_{f}|\mathcal H|\sqrt{T}\sqrt{\sum_{t=1}^{T}\sum_{i=1}^{N}\sum_{j\in\mathcal H_{i}}\mathbf E\left\|x_{ij,t}-x^{*}_{ij,t}\right\|^{2}}.
\end{align*}
\hfill$\blacksquare$

\end{appendix}

\begin{thebibliography}{10}
\providecommand{\url}[1]{#1}
\csname url@samestyle\endcsname
\providecommand{\newblock}{\relax}
\providecommand{\bibinfo}[2]{#2}
\providecommand{\BIBentrySTDinterwordspacing}{\spaceskip=0pt\relax}
\providecommand{\BIBentryALTinterwordstretchfactor}{4}
\providecommand{\BIBentryALTinterwordspacing}{\spaceskip=\fontdimen2\font plus
\BIBentryALTinterwordstretchfactor\fontdimen3\font minus
\fontdimen4\font\relax}
\providecommand{\BIBforeignlanguage}[2]{{%
\expandafter\ifx\csname l@#1\endcsname\relax
\typeout{** WARNING: IEEEtran.bst: No hyphenation pattern has been}%
\typeout{** loaded for the language `#1'. Using the pattern for}%
\typeout{** the default language instead.}%
\else
\language=\csname l@#1\endcsname
\fi#2}}
\providecommand{\BIBdecl}{\relax}
\BIBdecl


\bibitem{Deng_2022}
Z.~Deng, ``Distributed algorithm design for aggregative games of Euler-Lagrange systems and its application to smart grids,'' \emph{IEEE Trans. Cybern.}, vol.~52, no.~8, pp. 8315--8325, Aug. 2022.

\bibitem{Pan_2009}
Y.~Pan and L.~Pavel, ``Games with coupled propagated constraints in optical networks with multi-link topologies,'' \emph{Automatica}, vol.~45, no.~4, pp. 871--880, 2009.

\bibitem{Sun_2021}
C.~Sun and G.~Hu, ``Continuous-time penalty methods for Nash equilibriumseeking of a nonsmooth generalized noncooperative game,'' \emph{IEEE Trans. Autom. Control}, vol.~66, no.~10, pp. 4895--4902, Oct. 2021.

\bibitem{Henrion_2007}
R.~Henrion and W.~R\"{o}misch, ``On m-stationary points for a stochastic equilibrium problem under equilibrium constraints in electricity spot market modeling,'' \emph{Appl. Math.}, vol.~52, no.~6, pp. 473--494, 2007.


\bibitem{Watling_2006}
D.~Watling, ``User equilibrium traffic network assignment with stochastic travel times and late arrival penalty,'' \emph{Eur. J. Oper. Res.}, vol.~175, no.~3, pp. 1539--1556, Dec. 2006.

\bibitem{Cui_2021}
 S.~Cui, B.~Franci, S.~Grammatico, U.~V.~Shanbhag, and M.~Staudigl, ``A relaxed-inertial forward-backward-forward algorithm for stochastic generalized nash equilibrium seeking,'' in \emph{Proc. 60th IEEE Conf. Decis. Control}, Austin, Texas, USA, 2021, pp. 197–202.

\bibitem{Johnson_2013}
R.~Johnson and T.~Zhang, ``Accelerating stochastic gradient descent using predictive variance reduction,'' in \emph{Proc. Adv. Neural Inf. Process. Syst.}, Lake Tahoe, Nevada, USA, 2013, pp. 315--323.

\bibitem{Defazio_2014}
A.~Defazio, F.~R.~Bach, and S.~Lacoste-Julien, ``SAGA: A fast incremental gradient method with support for non-strongly convex composite objectives,'' in \emph{Proc. Int. Conf. Neural Inf. Process. Syst.}, Montreal, Canada, 2014, pp. 1646--1654.

\bibitem{Franci_2022}
 B.~Franci and S.~Grammatico, ``Stochastic generalized Nash equilibrium seeking in merely monotone games,'' \emph{IEEE Trans. Autom. Control}, vol.~67, no.~8, pp. 3905--3919, Aug. 2022.


\bibitem{Yu_2023}
R.~Yu, M.~Meng, and L.~Li, ``Distributed online learning for leaderless multi-cluster games in dynamic environments,'' \emph{IEEE Trans. Control. Netw. Syst.}, pp. 1--12, Dec. 2023. 

\bibitem{Su_2021}
L.~Su and N.~H.~Vaidya, ``Byzantine-resilient multiagent optimization,'' \emph{IEEE Trans. Autom. Control}, vol. 66, no. 5, pp. 2227--2233, May. 2021.

\bibitem{Karimireddy_2021}
S.~P.~Karimireddy, L.~He, and M.~Jaggi, ``Learning from history for Byzantine robust optimization,'' in \emph{Proc. Int. Conf. Mach. Learn.}, 2021, pp. 5311--5319.

\bibitem{Sahoo_2021}
S.~Sahoo, A.~Gokhale, and R.~K.~Kalaimani, ``Distributed online optimization with Byzantine adversarial agents,'' in \emph{Proc. IEEE Amer. Control Conf.}, Atlanta, GA, USA, 2022, pp. 222--227.

\bibitem{Wu_2020}
Z.~Wu, Q.~Ling, T.~Chen, and G.~B.~Giannakis, ``Federated variance-reduced stochastic gradient descent with robustness to Byzantine attacks,'' \emph{IEEE Trans. Signal Process.}, vol.~68, pp. 4583--4596, 2020.

\bibitem{Wu_2023} 
Z.~Wu, T.~Chen, and Q.~Ling, ``Byzantine-resilient decentralized stochastic optimization with robust aggregation rules,'' \emph{IEEE Trans. Signal Process.}, vol.~71, pp. 3179--3195, 2023.

\bibitem{Dong_2024}
X.~Dong, Z.~Wu, Q.~Ling and Z.~Tian, ``Byzantine-Robust Distributed Online Learning: Taming Adversarial Participants in An Adversarial Environment,'' \emph{IEEE Trans. Signal Process.}, vol.~72, pp. 235--248, 2024.

\bibitem{Gadjov_2023}
D.~Gadjov and L.~Pavel, ``An algorithm for resilient Nash equilibrium seeking in the partial information setting,'' \emph{IEEE Trans. Control Netw. Syst.}, vol.~10, no.~4, pp. 1645--1655, Dec. 2023.

\bibitem{Peng_2024}
J.~Zhao and P.~Yi, ``A Robust Distributed Nash Equilibrium Seeking Algorithm for Aggregative Games Under Byzantine Attacks,''  in \emph{Proc. Amer. Control Conf.}, Toronto, Canada, 2024, pp. 863-868.

\bibitem{Vaidya_2012_1}
N.~H.~Vaidya, L.~Tseng, and G.~Liang, ``Iterative approximate Byzantine consensus in arbitrary directed graphs,” in \emph{Proc. ACM Symp. Princ. Distrib. Comput.}, Madeira, Portugal, 2012, pp. 365--374.

\bibitem{Ravat_2011}
U.~Ravat and U.~V.~Shanbhag, ``On the characterization of solution sets of smooth and nonsmooth convex stochastic Nash games,'' \emph{SIAM J. Optim.}, vol.~21, no.~3, pp. 1168--1199, 2011.

\bibitem{Facchinei_2003}
F.~Facchinei and J.-S.~Pang, \emph{Finite-Dimensional Variational Inequalities and Complementarity Problem.} Berlin,~Germany:~Springer-Verlag, 2003.

\bibitem{Meng_2021}
M.~Meng, X.~Li, Y.~Hong, J.~Chen, and L.~Wang, ``Decentralized online learning for noncooperative games in dynamic environments,'' \emph{arXiv preprint arXiv:2105.06200}, 2021.

\bibitem{Meng_2022}
M.~Meng, X.~Li, and J.~Chen, ``Decentralized Nash equilibria learning for online game with bandit feedback,'' \emph{IEEE Trans. Autom. Control}, pp. 1--8, Dec. 2023.

\bibitem{Wang_2022}
R.~Wang, Y.~Liu, and Q.~Ling, ``Byzantine-resilient resource allocation over decentralized networks,'' \emph{IEEE Trans. Signal Process.}, vol.~70, pp. 4711--4726, 2022.

\bibitem{Zhu_2010}
M.~Zhu and S.~Martinez, ``Discrete-time dynamic average consensus,'' \emph{Automatica}, vol.~46, no.~2, pp. 322--329, 2010.

\bibitem{Wu_2010}
Z.~Wu, H.~Shen, T.~Chen, and Q.~Ling, ``Byzantine-resilient decentralized policy evaluation with linear function approximation,'' \emph{IEEE Trans. Signal Process.}, vol.~69, pp. 3839--3853, 2021.

\bibitem{Vaidya_2012}
N.~H.~Vaidya, ``Matrix representation of iterative approximate Byzantine consensus in directed graphs,'' 2012, \emph{arXiv}:1203.1888.


\bibitem{Bottou_2018}
L.~Bottou, F.~E.~Curtis, and J.·Nocedal, ''Optimization methods for large-scale machine learning,'' \emph{SIAM Rev.}, vol.~60, no.~2, pp. 223--311, 2018.
 
 
\bibitem{Nedic_2010}
A.~Nedi\'{c}, A.~Ozdaglar, and P.~A.~Parrilo, ``Constrained consensus and optimization in multi-agent networks,'' \emph{IEEE Trans. Autom. Control}, vol.~55, no.~4, pp. 922--938, Apr. 2010.

\bibitem{Nedic_2009}
A.~Nedi\'{c} and A.~Ozdaglar, ``Approximate primal solutions and rate analysis for dual subgradient methods,'' \emph{SIAM J. Optim.}, vol.~19, no.~4, pp. 1757--1780, Jan. 2009.

\end{thebibliography}

\vspace{-1.5em}

\end{document}